\documentclass[12pt,reqno]{amsproc}
\usepackage[T1]{fontenc}

\usepackage{amssymb}
\usepackage{amsmath}
\usepackage{amsthm}
\usepackage{amsfonts}
\usepackage{fullpage}
\usepackage{enumerate}
\usepackage{tikz-cd} 
\usepackage{comment}
\usepackage{float}
\usepackage{verbatim}
\usepackage{enumitem}
\usepackage{fancyhdr}
\usepackage{todonotes}
\usepackage[foot]{amsaddr}
\usepackage{svg}
\usepackage{multirow}
\usepackage{url}
\usepackage{textcomp}

\setlist[enumerate]{label*=\arabic*.} %subitems in \begin{enumerate} get subnumbering like 1.1, 1.2, ...

\usepackage[hidelinks]{hyperref}
\hypersetup{
    colorlinks,
    citecolor=magenta,
    filecolor=magenta,
    linkcolor=red,
    urlcolor=blue
}

\newcommand{\nc}{\newcommand}
\newcommand{\rc}{\renewcommand}

\nc\on{\operatorname}

\nc{\Spec}{\on{Spec}}
\nc{\Proj}{\on{Proj}}
\nc\Id{\on{Id}}
\nc{\Hom}{\on{Hom}}
\nc{\res}{\on{res}}
\nc{\Cone}{\on{Cone}}
\nc{\Conv}{\on{Conv}}
\nc{\tr}{\on{tr}}
\nc{\Lie}{\on{Lie}}
\nc{\Der}{\on{Der}}
\nc{\Aut}{\on{Aut}}
\nc{\End}{\on{End}}
\nc{\Ad}{\on{Ad}}
\nc{\ad}{\on{ad}}
\nc{\Ext}{\on{Ext}}
\nc{\Mor}{\on{Mor}}
\nc{\Tor}{\on{Tor}}
\nc{\Gal}{\on{Gal}}
\nc{\ab}[1]{{#1}^{\on{ab}}}
\nc{\Nm}{\on{Nm}}

\renewcommand{\d}{\text{d}}

\nc{\lan}{\langle}
\nc{\ran}{\rangle}
\nc{\la}{\leftarrow}
\nc{\ra}{\rightarrow}

\nc{\injto}{\hookrightarrow}
\nc{\surjto}{\twoheadrightarrow}

\nc{\di}{\d_{x_i}}
\rc{\dj}{\d_{x_j}}

\nc{\dol}{\partial}

\nc{\pder}[1]{\frac{\dol}{\dol {#1}}}
\nc{\pderh}[2]{\frac{\dol^{#2}}{\dol {#1}^{#2}}}

\nc{\bs}{\bigskip}
\nc{\ms}{\medskip}

\nc{\noi}{\noindent}
\nc{\tn}{\textnormal}
\nc{\tb}{\textbf}
\nc{\mb}{\mathbf}
\nc{\dsp}{\displaystyle}
\nc{\tc}{\textcolor}

\rc{\AA}{\mathbb{A}}
\nc{\CC}{\mathbb{C}}
\nc{\FF}{\mathbb{F}}
\nc{\KK}{\mathbb{K}}
\nc{\NN}{\mathbb{N}}
\nc{\PP}{\mathbb{P}}
\nc{\QQ}{\mathbb{Q}}
\nc{\RR}{\mathbb{R}}
\rc{\SS}{\mathbb{S}}
\nc{\ZZ}{\mathbb{Z}}

\nc{\CM}{\mathcal{M}}

\nc{\mf}{\mathfrak}
\rc{\sf}{\mathsf}
\nc{\mc}{\mathcal}

\nc{\fra}{\mathfrak{a}}
\nc{\frb}{\mathfrak{b}}
\nc{\frp}{\mathfrak{p}}

\nc{\vv}[2]{\begin{bmatrix} #1\\ #2 \end{bmatrix}}
\nc{\vvv}[2]{\begin{bmatrix} #1\\ #2 \end{bmatrix}}
\nc{\mat}[4]{\begin{bmatrix} #1 & #2 \\ #3 & #4 \end{bmatrix}}
\nc{\mmat}[9]{\begin{bmatrix} #1 & #2 & #3 \\ #4 & #5 & #6 \\ #7 & #8 & #9 \end{bmatrix}}
\nc{\ptl}[2]{\frac{\d #1}{\d #2}}
\nc{\congr}[3]{#1 \stackbin{#3}{\equiv} #2}

\theoremstyle{definition}

\rc{\theq}{\thesubsection.\Alph{q}}

\rc{\thedis}{\thesubsection.\Alph{dis}}
\newtheorem{theorem}{Theorem}
\newtheorem{lemma}{Lemma}
\newtheorem{definition}{Definition}
\newtheorem{proposition}{Proposition}
\newtheorem{corollary}{Corollary}

\numberwithin{theorem}{section}
\numberwithin{proposition}{section}
\numberwithin{definition}{section}
\numberwithin{lemma}{section}
\numberwithin{corollary}{section}

\begin{document}

\title{Cohomology of $p$-adic fields and \\
Local class field theory}
\author[U. Lim]{Uzu Lim}
\address{Mathematical Institute, 
University of Oxford, Radcliffe Observatory, Andrew Wiles Building, Woodstock Rd, Oxford OX2 6GG}
\email{lims@maths.ox.ac.uk}

\maketitle

\begin{abstract}
    In this expository article, we outline a basic theory of group (co)homology and prove a cohomological formulation of the Local Reciprocity Law:
    $$\Gal(L/K)^{\on{ab}} \cong H_T^{-2}(\Gal(L/K),\ZZ) \cong H_T^{0}(\Gal(L/K),L^\times) \cong \frac{K^\times}{\Nm_{L/K}(L^\times)}$$
    We first recall basic facts about local fields and homological algebra. Then we define group (co)homology, Tate cohomology, and furnish a toolbox. The Local Reciprocity Law is proven in an abstract cohomological setting, then applied to the case of local fields.
\end{abstract}

\tableofcontents
\pagebreak

The primary reference for this article is \cite{milne_cft}. A good reference for an introductory course in algebraic number theory is \cite{milne_ant}. A good reference for the theory of local fields is \cite{serre_lf}. This article was written following a course by Jeehoon Park in Postech in 2017.

\section{Preliminaries}

Throughout the article, a function $f$ will be also written $f=(x \mapsto f(x))$.

We recall some facts from algebraic number theory and homological algebra here.

\subsection{Local Fields}

A local field is a discretely valued field which is complete and has finite residue field. The residue field of a local field $K$ is $\mathcal O_K / \mf m_K$ where $\mf m_K$ is the unique maximal ideal of $\mc O_K$. Here onwards we write $p$ for the residue characteristic of a local field. Residue fields of local fields $L,K$ will be denoted by $\ell, k$. The group of units of $K$ will be denoted by $U_L$.

\begin{proposition}
A nonarchimedian local field is either a finite extension of $\QQ_p$ or $\FF_p((t))$.
\end{proposition}

\begin{proposition}
For a finite Galois extension $L/K$, there is a surjection
$$\Gal(L/K) \ra \Gal(\ell / k) \ra 0$$
In particular, for unramified extension $L/K$, there is an element of $\Gal(L/K)$ that induces Frobenius map on the residue fields, also referred to as a Frobenius element and denoted $\on{Frob}_{L/K}$. In particular, the Galois group is cyclic.
\end{proposition}

\begin{lemma}
    Suppose $L/K$ is a finite extension of local fields with residue fields $\ell/k$. If $L=K(\alpha)$ with $\alpha \in \mc O_L$, then $\ell = k(\bar \alpha)$.
\end{lemma}
\begin{proof}
    As $L=K(\alpha)=K[\alpha]$, each element in $\ell$ is represented by a polynomial of $\alpha$ in $K$, which is a rational function of $\alpha$ in $\mc O_K$ if we factor out the prime power (e.g. $\frac14 + \frac18 \alpha + \alpha^2 = \frac{2+\alpha + 8\alpha^2}{8}$). Thus $\ell = k(\bar \alpha)$.
\end{proof}

\begin{lemma}
    Suppose $L/K$ is a Galois extension of local fields. Then each element of $\Gal(L/K)$ preserves valuation.
\end{lemma}
\begin{proof}
    For $\alpha \in L$, the Newton polygon of its minimal polynomial $f$ describes factoring into products of linear factors with same valuation. This implies that conjugates of each element are all of the same valuation.
\end{proof}

\begin{proposition}
Suppose $K$ is a local field and $k$ its residue field, where both are perfect fields. Suppose $L/K$ is an algebraic extension and $\ell/k$ the corresponding extension of residue fields. Then:
\begin{enumerate}
    \item There is a bijection
    $$\{\text{Finite unramified extensions of $K$ in $L$}\}\leftrightarrow \{\text{Finite extensions of $k$ in $\ell$}\} $$
    given by taking a finite unramified extension to its residue field.
    \item If $K' \leftrightarrow k', K'' \leftrightarrow k''$, then $K' \subseteq K'' \iff k' \subseteq k''$
    \item If $K' \leftrightarrow k',$ $K'/K$ is Galois iff $k'/k$ is Galois, in which case
    $$\Gal(K'/K) \cong \Gal(k'/k)$$
\end{enumerate}
\end{proposition}
\begin{proof}
    We first show that every finite intermediate extension $k'$ is a residue field of some unramified extension. Let $k'=k(\alpha_0)$ and $f_0 \in k[x] \subset \ell[x]$ be the (separable) minimal polynomial of $\alpha_0$. Let $f \in K[x] \subset L[x]$ be any lift of $f_0$, which by Hensel's lemma has a simple root $\alpha$ that lifts $\alpha_0$. $K(\alpha)/K$ is unramified since $f(\alpha)=0$ and $f$ is irreducible (since its reduction $f_0$ is irreducible), implying $[K(\alpha):K] = \deg f = \deg f_0 = [k(\alpha_0):k]$. By the previous lemma, residue field of the unramified extension $K(\alpha)$ is $k(\alpha_0) = k'$ as desired. 
    
    Suppose $K', K''$ are finite unramified intermediate extensions with residue field both $k'$; we will show that $K'=K''$. $K'K''$ also has residue field $k'$ since if $K' = K(\alpha), K = K(\beta)$, then $K'K'' = K(\alpha, \beta)$, $k'=k(\bar \alpha) = k(\bar \beta) \implies \bar \beta \in k(\bar \alpha)$ so that by the previous lemma the residue field of $K'K''$ is $k'(\bar \beta) = k(\bar \alpha)(\bar \beta) = k(\alpha) = k'$. By discriminant calculation (details on page 128 of Milne's ANT textbook), $K'K''/K$ is unramified. This implies
    $$[K'K'':K] = [k':k] = [K':K] = [K'':K]$$
    and thus $K'K''= K'$ and thus $K' = K'K'' = K''$. This establishes injectivity of the correspondence.
    
    (2) is `obvious'.
    
    (3) Suppose $K'/K$ is unramified and Galois. As an element of $K'/K$ preserves valuation, it preserves $\mc O_{K'}$ and $\mf p'$ (valuation ring and maximal ideal of $K'$). Thus a map on $k'$ is induced and there is a homomorphism $\Gal(K'/K) \rightarrow \Gal(k'/k)$. We are to establish that this is an isomorphism, and 
\end{proof}

\begin{proposition}
The maximal unramified extension of a local field is obtained by adjoining all $m$th roots of unity where $(p,m)=1$.
\end{proposition}

We refer the readers to J.S. Milne's book on Algebraic Number Theory for proofs of the above assertions.

\begin{proposition}\label{local solvable}
For a finite Galois extension $L/K$ of local fields, $\Gal(L/K)$ is solvable.
\end{proposition}
\begin{proof}
Denoting by $\ell,k$ the residue fields of $L$ and $K$, we have exact sequences
\begin{align*}
& 0 \ra I_{L/K} \ra \Gal(L/K) \ra \Gal(\ell / k) \\
& 0 \ra I_{L/K}^{\on{wild}} \ra I_{L/K} \ra I_{L/K}^{\on{tame}} \ra 0
\end{align*}
Here $\Gal(\ell/k)$ is cyclic, $I_{L/K}^{\on{tame}}$ is abelian, $I_{L/K}^{\on{wild}}$ is a pro-$p$ group. If $H \trianglelefteq G$ and $G/H$ are both solvable, then $G$ is solvable, and therefore we deduce that $\Gal(L/K)$ is solvable.
\end{proof}

\subsection{Hom and Tensor}

We recall some algebraic facts regarding $\Hom_R(A,B)$ and $A \otimes_R B$.

\begin{proposition}
    Suppose $M,N$ are left $R$-modules. Then
    \begin{enumerate}
        \item If $M$ is $(R,S)$-bimodule, $\Hom_R(M,N)$ is a \emph{left} $S$-module.
        \item If $N$ is $(R,S)$-bimodule, $\Hom_R(M,N)$ is a \emph{right} $S$-module.
    \end{enumerate}
\end{proposition}
Note that a $(R,S)$-bimodule satisfies compatibiity between two actions: $r(ms) = (rm)s$.

\begin{proposition}\label{hom tensor distributive}
Suppose $A_i, B_j$ are $R$-modules indexed by $i \in I, j \in J$. Then
\begin{align*}
& \Hom_R(\bigoplus_{i \in I} A_i , B) \cong \prod_{i \in I} \Hom_R (A_i, B) \\
& \Hom_R(A,\prod_{j \in J}B_j) \cong \prod_{j \in J} \Hom_R(A, B_j) \\
& (\bigoplus_{i \in I} A_i ) \otimes_R ( \bigoplus_{j \in J} B_j ) \cong \bigoplus_{i \in I} \bigoplus_{j \in J} (A_i \otimes_R B_j)
\end{align*}
\end{proposition}

\begin{proposition}[Adjoint Associativity] \label{adjoint associativity}
Suppose $A_1$ is a left $R_1$-module, $A_2$ is a $(R_2,R_1)$-bimodule, and $A_3$ is a left $R_2$-module. Then
\begin{align*}
& \Hom_{R_1}(A_1 , \Hom_{R_2}(A_2, A_3)) \cong \Hom_{R_2}(A_2 \otimes_{R_1} A_1 , A_3) \\
& f \mapsto( a \otimes b \mapsto f(b)(a))
\end{align*}
\end{proposition}
There is a simpler-looking relation involving right modules for which we don't switch order of the modules, but we will be using left modules mostly.

\begin{proposition}
Given an exact sequence of $R$-modules
$$0 \ra A \ra B \ra C \ra 0$$
the following are exact:
\begin{align*}
0 \ra &\Hom_R(M,A) \ra \Hom_R(M,B) \ra \Hom_R(M,C) \\
0 \ra &\Hom_R(C,M) \ra \Hom_R(B,M) \ra \Hom_R(A,M) \\
&M \otimes_R A \ra M \otimes_R B \ra M \otimes_R C \ra 0
\end{align*}
\end{proposition}
In the above proposition, right-exactness was lost for functor $\Hom_R(M,-)$ and $\Hom_R(-,M)$ and left-exactness was lost for functor $M \otimes_R$. If $M$ preserves exactness in each situation, $M$ is respectively referred to as \emph{projective, injective,} and \emph{flat} modules.

\begin{proposition}\label{projective from direct sum}
If $\bigoplus_{i \in I} P_i$ is projective, then each $P_i$ is projective.
\end{proposition}
\begin{proof}
If $A \twoheadrightarrow B$, then $\Hom_R (\bigoplus_i P_i, A) \twoheadrightarrow \Hom_R (\bigoplus_i P_i, B)$, then by Proposition~\ref{hom tensor distributive}, $\prod_i \Hom_R (P_i, A) \twoheadrightarrow \prod_i \Hom_R (P_i, B)$, and thus $\Hom_R (P_i,A) \twoheadrightarrow \Hom_R (P_i,B)$. Thus each $P_i$ is projective.
\end{proof}

\begin{proposition}
A free module is projective.
\end{proposition}
\begin{proof}
When there is a surjection $f: B \rightarrow C$ and $\varphi: \bigoplus R \rightarrow C$, then we can construct $\tilde \varphi: \bigoplus R \rightarrow B$ by lifting images of basis elements under $\varphi$ arbitrarily through $f$.
\end{proof}

\begin{proposition}
A free module is flat.
\end{proposition}
\begin{proof}
If $A \hookrightarrow B$, then $A \otimes_R (\bigoplus_{i \in I} R) \cong \bigoplus_{i \in I} A \hookrightarrow \bigoplus_{i \in I} B \cong B \otimes_R (\bigoplus_{i \in I} R)$.
\end{proof}

\begin{proposition}
For a $R$-module $P$, the following are equivalent:
\begin{enumerate}
\item $P$ is a projective module
\item Any surjection $A \rightarrow P \rightarrow 0$ splits
\item $P$ is a direct summand of a free module
\end{enumerate}
\end{proposition}
\begin{proof}
(1 $\implies$ 2) If $P$ is projective and $A \rightarrow P$ surjects, then we can construct a lift of the identity map $P \rightarrow P$ and obtain $P \rightarrow A$ whose composition with $A \rightarrow P$ is $\on{id}_P$. This is a splitting.

(2 $\implies$ 3) Construct a surjection $\bigoplus R \rightarrow P$ by taking generators of $P$ and apply the hypothesis.

(3 $\implies$ 1) If $P$ is a direct summand of a free module, then by Proposition~\ref{projective from direct sum}, $P$ is projective.
\end{proof}

\begin{proposition}
A projective module is flat.
\end{proposition}
\begin{proof}
Suppose $P$ is a projective $R$-module and suppose $P \oplus M$ is a free module. Then $\otimes_R (P \oplus M)$ is left-exact, and thus $\otimes_R P$ is left-exact.
\end{proof}

\begin{proposition}
Any $R$-module has a free (and thus projective) resolution.
\end{proposition}
\begin{proof}
For a $R$-module $M$, take its set of generators to get a surjection $\bigoplus_{i \in I_1} R \rightarrow M \rightarrow 0$. Now take kernel of this surjection and apply the same process to get $\bigoplus_{i \in I_2} R \ra \bigoplus_{i \in I_1} R \rightarrow M \rightarrow 0$, and continue inductively.
\end{proof}

\begin{proposition}
A module $I$ is injective iff $\Hom_R(-,I)$ is an exact functor on any exact sequence.

A module $P$ is projective iff $\Hom_R(P,-)$ is an exact functor on any exact sequence.
\end{proposition}
\begin{proof}
Suppose we have exact sequence
$$A_m \ra A_{m-1} \ra \cdots \ra A_1$$
We want
$$\Hom(A_1, I)\ra \cdots \ra \Hom(A_m,I)$$
to be exact. This is equivalent to lifting $f: A_n \ra I$ to $\tilde f:A_{n-1} \ra I$ in
\[\begin{tikzcd}
A_{n-1} \arrow[dr, dashed] & A_n \arrow[l,"\alpha_n"] \arrow[d,"f"] & A_{n+1} \arrow[l,"\alpha_{n+1}"] \\
&I
\end{tikzcd}\]
where $f \circ \alpha_{n+1} = 0$. This is done by defining $\tilde f$ first on $\on{im} \alpha_n$ as $\tilde f (\alpha_n(t)) = f(t)$. This is shown to be well-defined, and by injectivity this map on $\on{im} \alpha_n$ lifts to $A_{n-1}$. We obtain the result for projective modules by dual argument (switching arrow directions).
\end{proof}

The following modules will occur frequently later.
\begin{definition}
For $H\le G$ a subgroup and $M$ a $H$-module, the coinduced module and induced module are respectively $G$-modules
$$\Hom_{\ZZ[H]}(\ZZ[G],M), \ZZ[G] \otimes_{\ZZ[H]} M$$
Coinduced module is sometimes written $\on{Ind}_H^G(M)$. When $H=1$, it's also written $\on{Ind}^G(M)$.
\end{definition}

Induced and coinduced modules are isomorphic when dealing with a subgroup of finite index:
\begin{proposition}
If $M$ is a $G$-module and $H\le G$ is a finite-index subgroup,
\begin{align*}
& \Hom_{\ZZ[H]} (\ZZ[G], M) \cong \ZZ[G]\otimes_{\ZZ[H]} M \\
& \varphi \mapsto \sum_{j=1}^n g_j \otimes \varphi(g_j^{-1})
\end{align*}
where $g_1, \cdots g_n$ are the distinct coset representatives of $H$ in $G$.
\end{proposition}
\begin{proof}
The map is well-defined because if another coset representative $g_j h$ is chosen in place of $g_j$, then $g_j h \otimes \varphi((g_j h)^{-1}) = (g_j h) \otimes h^{-1} \varphi(g_j^{-1}) = g_j \otimes \varphi(g_j^{-1})$. Every element in $\ZZ[G] \otimes_{\ZZ[H]} M$ can be uniquely written as a sum $\sum_{k=1}^n g_k \otimes m_k$, and thus it is image of $\varphi$ where $\varphi$ is defined by $\varphi(g_k^{-1} h) = h^{-1} m_k$. This establishes surjectivity and injectivity.
\end{proof}

\subsection{Ext and Tor Functors}

Given a left exact additive functor $\mc F:\textbf{Mod}_R \rightarrow \mc C$  where $\mc C$ is an abelian category, we get \emph{derived functors} $\mc F_n$ such that whenever we are given
$$ 0 \ra A \ra B \ra C \ra 0$$
we get a long exact sequence
\begin{align*}
0 & \ra \mc F(A) \ra \mc F(B) \ra \mc F(C) \ra \\
& \ra \mc F_1(A) \ra \mc F_1(B) \ra \mc F_1(C) \ra \\
& \ra \mc F_2(A) \ra \mc F_2(B) \ra \mc F_2(C) \ra \cdots
\end{align*}
which precisely fills in the missing right-exactness of $\mc F$. Analogous theory exists for when $\mc F$ is right exact.

Ext is derived functor for $\Hom_R(A,-)$ and $\Hom_R(-,B)$ and Tor is derived functor for $\otimes_R B$.

\begin{definition}
Given module $A$, take its projective resolution:
$$\cdots \ra P_1 \ra P_0 \ra A \ra 0$$
and take functor $\Hom_R(-,B)$ to get a chain complex (disregard $\Hom_R(A,B)$ here)
$$0 \ra \Hom_R (P_0,B) \ra \Hom_R (P_1,B) \ra \cdots$$
Then $\Ext_R^n(A,B)$ is defined as the cohomology of this complex at $\Hom_R(P_n,B)$.

By taking functor $\otimes_R B$ we also get a chain complex
$$\cdots \ra P_1 \otimes_R B \ra P_0 \otimes_R B \ra 0$$
Then $\Tor_n^R(A,B)$ is defined as the homology of this complex at $P_n \otimes_R B$.
\end{definition}

We need to check several facts to ensure that Ext and Tor induce a (co)homology long exact sequence. The following summarizes the case for Ext.
\begin{enumerate}
\item (Functoriality of Ext $\&$ Independence from choice of lifts) Firstly $f: A \ra B$ lifts to projective resolution inductively:
\[ \begin{tikzcd}
\cdots \arrow[r] & P_1 \arrow[r]\arrow[d, dotted, "f_1"] &P_0 \arrow[r]\arrow[d, dotted, "f_0"] & A \arrow[r]\arrow[d, "f"] & 0 \\
\cdots \arrow[r] & P'_1 \arrow[r] & P'_0 \arrow[r] & B \arrow[r]& 0
\end{tikzcd}
\]
The composed map $P_0 \ra A \ra B$ is lifted to $P_0 \xrightarrow{f_0} P_0'$ by projectivity of $P_0$. Then $P_1 \ra P_0 \ra P_0'$ is lifts to $P_1 \xrightarrow{f_1} P_1'$ because $P_1 \ra P_0'$ is a map to $\on{ker}(P_0' \ra B) = \on{im}(P_1' \ra P_0')$. We continue this way.

We now use chain homotopy to show that the map of homology is independent of lifts $f_j$ chosen. When two lifts $f_j, f_j'$ are given, subtract them and consider homology induced by $f_j - f_j'$. It remains to show that homology induced by lift of a zero map $A \ra B$ is 0 on homologies. To do this we construct chain homotopy $s_n: P_n \ra P_n'$ with $sd + ds = f$. We let $s_{-1}:A \ra P_0'$ to be a zero map and $s_0$ lifts directly by projectivity. Construction of $s_1$ is done by lifting
\[
\begin{tikzcd}
& P_1 \arrow[dl, dashed, swap, "s_1"] \arrow[d, "f_1 - s_0 d_1"] \\
P_2' \arrow[r] & \on{im}d_2' \arrow[r] & 0
\end{tikzcd}
\]
and we can continue similarly.
\item (Independence from choice of resolution) We use functoriality: consider
\[ \begin{tikzcd}
P_\bullet \arrow[r]\arrow[d] & A \arrow[r]\arrow[d,"\text{id}"] & 0 \\
P'_\bullet \arrow[r]\arrow[d] & A \arrow[r]\arrow[d,"\text{id}"] & 0 \\
P_\bullet \arrow[r] & A \arrow[r]& 0
\end{tikzcd}
\]
\item (Homology long exact sequence for $\Hom_R (-,D)$) Given $0 \ra L \ra M \ra N \ra 0$, we are to construct long exact sequence $0 \ra \Hom_R (N,D) \ra \Hom_R(M,D) \ra \Hom_R(L,D) \ra \Ext_R^1(N,D) \ra \cdots$. To do this, we simply find an adequate short exact sequence of chain complex, from which a homology long exact sequence comes. Ideally, we have projective resolutions 
\begin{align*}
P_\bullet \ra L \ra 0\\
P_\bullet' \ra M \ra 0\\
P_\bullet'' \ra N \ra 0
\end{align*}
such that
\[
\begin{tikzcd}
& \vdots & \vdots & \vdots & \\
0 \arrow[r]& \Hom_R (P_1'',D) \arrow[r]\arrow[u]& \Hom_R (P_1',D) \arrow[r]\arrow[u]& \Hom_R (P_1,D) \arrow[r]\arrow[u] & 0 \\
0 \arrow[r]& \Hom_R (P_0'',D) \arrow[r]\arrow[u]& \Hom_R (P_0',D) \arrow[r]\arrow[u]& \Hom_R (P_0,D) \arrow[r]\arrow[u] & 0 \\
& 0\arrow[u] & 0\arrow[u] & 0\arrow[u] &
\end{tikzcd}
\]
with rows exact. To do this, we actually just take any resolutions $P_\bullet,P_\bullet''$ and then let $P_\bullet' = P_\bullet \oplus P_\bullet''$. Right exactness of rows are ensured by splitness of $0\ra P_\bullet'' \ra P_\bullet'' \oplus P_\bullet \ra P_\bullet \ra 0$.
\item (Homology long exact sequence for $\Hom_R (D,-)$) This time around, we are to find long exact sequence 
$$0 \ra \Hom_R (D,L) \ra \Hom_R (D,M) \ra \Hom_R (D,N) \ra \Ext_R^1(D,L) \ra \cdots$$
To do so, we work analogously as above, but we just find projective resolution for $D$ and apply $\Hom_R (-,L),\Hom_R (-,M),\Hom_R (-,N)$. The resulting sequence of chain complexes is exact due to projectivity of $P_n$, and as before we extract a homology long exact sequence.
\end{enumerate}

\begin{proposition}
\begin{align*}
& \Ext_R^n(\bigoplus_{i \in I} A_i, B) \cong \prod_{i \in I} \Ext_R^n(A_i, B) \\
& \Ext_R^n(A, \prod_{j \in J}B_j) \cong \prod_{j \in J} \Ext_R^n(A, B_j) \\
& \Tor_n^R(\bigoplus_{i \in I} A_i, \bigoplus_{j \in J} B_j) \cong \bigoplus_{i \in I} \bigoplus_{j \in J} \Tor_n^R(A_i, B_j)
\end{align*}
\end{proposition}
\begin{proof}
These follow from Proposition~\ref{hom tensor distributive}.
\end{proof}

\begin{proposition}
$A$ is projective $\iff$ $\forall B, \Ext_R^1(A,B)=0$ $\iff$ $\forall B, n \ge 1, \Ext_R^n(A,B)=0$.
\end{proposition}

\begin{proposition}
$A$ is flat $\iff$ $\forall B, \Tor_1^R(A,B)=0$ $\iff$ $\forall B, n\ge 1, \Tor_n^R(A,B)=0$.
\end{proposition}

The following exact sequences will be used frequently:
\begin{align*}
& 0 \ra I_G \ra \ZZ[G] \xrightarrow{\epsilon} \ZZ \ra 0 \\
& 0 \ra \ZZ \xrightarrow{\mu} \ZZ[G] \ra J_G \ra 0
\end{align*}
where $\epsilon: \sum_{g \in G} n_g g \mapsto \sum_{g \in G} n_g$ and $\mu: n \mapsto n \sum_{g \in G} g$ (defined only for a finite group $G$), and $I_G:= \ker \epsilon, J_G := \on{coker}\mu$.

\begin{proposition}
Both $I_G$ and $J_G$ are free modules and we have
$$I_G \oplus \ZZ \cong \ZZ[G] \cong J_G \oplus \ZZ$$
\end{proposition}
\begin{proof}
Define map $\ZZ[G] \ra I_G \oplus \ZZ$ by $a \mapsto (a-\epsilon(a), \epsilon(a))$ and define map $J_G \oplus \ZZ \ra \ZZ[G]$ by $([a],n) \mapsto a + (n-a_1) \Nm_G$ where $a = \sum_{g\in G} a_g g$ and $\Nm_G = \sum_{g \in G} g$. These are easily checked to be isomorphisms. $I_G$ is free with generators $g-1$ and $J_G$ is free with generators $[g]$.
\end{proof}

\begin{proposition}
For $G$-module $M$, the following sequences are exact:
\begin{align*}
& 0 \ra I_G \otimes M \ra \ZZ[G] \otimes M \ra M \ra 0 \\
& 0 \ra M \ra \ZZ[G] \otimes M \ra J_G \otimes M \ra 0
\end{align*}
\end{proposition}
\begin{proof}
The first sequence is obtained from applying $\otimes_\ZZ M$ to 
$$0 \ra I_G \ra \ZZ[G] \ra \ZZ \ra 0$$
and it remains exact because $\Tor_1^\ZZ(\ZZ,M)=0$ as $\ZZ$ is free over $\ZZ$.

The second sequence is obtained from applying $\otimes_Z M$ to
$$0 \ra \ZZ \ra \ZZ[G] \ra J_G \ra 0$$
and it remains exact because $\Tor_1^\ZZ(J_G,M) = 0$ as $J_G$ is free over $\ZZ$.
\end{proof}

We give an explicit computation of the connecting homomorphism of long exact sequences associated to Ext and Tor.

\begin{proposition}\label{ext tor delta}
Given exact sequence
$$0 \ra L \ra M \ra N \ra 0$$
suppose Ext and Tor are defined using a projective resolution $P_\bullet \rightarrow A \ra 0$. Then the connecting homomorphisms of Ext and Tor are given by:
\begin{align*}
\delta: & \Ext_R^r(A,N) \ra \Ext_R^{r+1}(A,L) \\
& [\varphi] \mapsto [\d \tilde \varphi]
\end{align*}
where $\tilde \varphi \in \Hom_R(P_r,M)$ is a lift constructed using projectivity of $P_r$, and
\begin{align*}
\delta: & \Tor_r^R(A,N) \ra \Tor_{r-1}^R(A,L) \\
& \left[ \sum_p p \otimes a_p \right] \mapsto \left[ \sum_p \d p \otimes \tilde{a}_p \right]
\end{align*}
where $\tilde{a}_p$ is a lift of $a_p$ through surjection $M \ra N$.
\end{proposition}

\pagebreak
\section{Cohomology}

Here we introduce algebraic devices called group cohomology, homology, and Tate cohomology. Group cohomology is a type of Ext group, group homology is a type of Tor group, and Tate cohomology is a concatenation of group cohomology and homology.

\begin{definition}
For a group $G$, a module over group $G$ is a module over ring $\ZZ[G]$. In other words, there is an action $G \times M \rightarrow M$ that satisfies
\begin{align*}
& g \cdot (a+b) = g \cdot a + g \cdot b \\
& gh \cdot a = g \cdot (h \cdot a)
\end{align*}
\end{definition}

It's natural to consider a maximal submodule fixed by $G$ and a maximal quotient module fixed by $G$. These are called \emph{invariant} and \emph{coinvariant} modules: $M^G$ and $M_G$, defined as
\begin{align*}
& M^G := \{a \in M | \forall g \in G, g \cdot a = a\} \\
& M_G := M / I_G M
\end{align*}
It's evident that $M^G$ is the maximal submodule fixed by $G$. $M_G$ is the maximal quotient module fixed by $G$ because if $M/A$ is fixed by $G$, then $\forall a \in M, g \in G, g \cdot a - a \in A \iff \forall a, g, (g-1) a \in A \iff I_G M \subseteq A$.

\subsection{Group (co)homology and Tate cohomology}

Group cohomology and homology study when invariant and coinvariant parts associated to an exact sequence retains exactness. To each exact sequence 
\begin{align*}
0 \ra A \ra B \ra C \ra 0
\end{align*}
we get exact sequences:
\begin{align*}
& 0 \ra A^G \ra B^G \ra C^G \ra H^1(G,A) \ra H^1(G,B) \ra H^1(G,C) \ra H^2(G,A) \ra \cdots \\
& \cdots \ra H_2(G,C) \ra H_1(G,A) \ra H_1(G,B) \ra H_1(G,C) \ra A_G \ra B_G \ra C_G \ra 0
\end{align*}
so that if we have $H^1(G,A)=0$ then the sequence of invariant modules is exact and if we have $H_1(G,C)=0$ then the sequence of coinvariant modules is exact.

In fact, group cohomology and homology are respectively Ext and Tor functors, as we will see:
$$H^n(G,A) := \Ext_{\ZZ[G]}^n(\ZZ,A), H_n(G,A) := \Tor^{\ZZ[G]}_n(\ZZ,A)$$
In other words, we can study invariant and coinvariant modules using Hom and $\otimes$:
\begin{proposition}
Let $G$ act on $\ZZ$ trivially. Then,
\begin{align*}
& M^G \cong \Hom_{\ZZ[G]}(\ZZ,M) \text{ by } m \mapsto (1 \mapsto m) \\
& M_G \cong \ZZ \otimes_{\ZZ[G]}M \text{ by } [m] \mapsto 1 \otimes m 
\end{align*}
\end{proposition}
\begin{proof}
$\varphi \in \Hom_{\ZZ[G]}(\ZZ,M) \iff \forall g\in G, \varphi(g \cdot n) = g \cdot \varphi(n) \iff \forall g \in G, \varphi(1) = g \cdot \varphi(1)$. 

The second map is well-defined because $1 \otimes (gm-m) = 1 \otimes gm - 1 \otimes m = (g \cdot 1) \otimes m - 1 \otimes m = 0$. It is surjective because every element in $\ZZ \otimes_{\ZZ[G]} M$ can be written as $1 \otimes m$. It is injective because such expression is unique.
\end{proof}

Since group cohomology is an Ext functor, it helps to have an explicit resolution. We define two explicit resolutions that strongly resemble singular cochain construction in topology:
$$C_\bullet(G,M) \rightarrow \ZZ \rightarrow 0 \text{ and } \tilde C_\bullet(G,M) \rightarrow \ZZ \rightarrow 0 $$
We let $C^r(G,M) = \Hom_{\ZZ[G]}(C_r(G), M)$ and $\tilde C^r(G,M) = \Hom_{\ZZ[G]}(\tilde C_r(G), M)$, where
\begin{align*}
C_r(G) &:= \bigoplus_{\gamma \in G^{r}} \ZZ[G] \cdot \gamma \\
\tilde C_r(G) &:= \bigoplus_{\gamma \in G^{r+1}} \ZZ \cdot \gamma \text{   } ; g \cdot (g_0, \cdots g_r) = (gg_0, \cdots gg_r)
\end{align*}
($C_r(G)$ is simply a free $\ZZ[G]$-module, and the $G$-action on $\tilde C_r(G)$ defined above extends by linearity) Also we define maps:
\begin{align*}
d_r &: C_r(G) \rightarrow C_{r-1}(G)\\
d_r(g_1, \cdots g_r) &:= g_1 \cdot (g_2, \cdots g_r) + \left( \sum_{k=1}^{r-1} (-1)^k (g_1, \cdots g_k g_{k+1}, \cdots g_r) \right) + (-1)^r (g_1, \cdots g_{r-1}) \\
\tilde d_r &: \tilde C_r(G) \rightarrow \tilde C_{r-1}(G) \\
\tilde d_r(g_0, \cdots g_r) &:= \sum_{k=0}^r (-1)^k (g_0, \cdots \hat g_k, \cdots g_r)
\end{align*}
where $\hat g_k$ indicates omission of index. This gives rise to maps $d^r$ and $\tilde d^r$:
\begin{align*}
d^r &: C^r(G,M) \rightarrow C^{r+1}(G,M) \\
d^r (\varphi) &= \varphi \circ d_r  \\
\tilde d^r &: \tilde C^r(G,M) \rightarrow \tilde C^{r+1}(G,M) \\
\tilde d^r (\varphi) &= \varphi \circ \tilde d_r 
\end{align*}
Then we claim that $C^\bullet(G,M) \rightarrow \ZZ \rightarrow 0$ and $\tilde C^\bullet(G,M) \rightarrow \ZZ \rightarrow 0$ are both projective (free) resolutions of $\ZZ$. Before we prove this, we show that the two resolutions are isomorphic:
\begin{proposition}
There are $G$-module isomorphisms $\Phi_r : C_r(G) \rightarrow \tilde C_r(G)$ for $r \ge 0$ given by
$$\Phi_r : g \cdot (g_1, \cdots g_r) \mapsto (g, g g_1, g g_1g_2, \cdots (g g_1 \cdots g_r))$$
(definition extends by linearity) such that $\Phi_{r-1} \circ d_r = \tilde d_r \circ \Phi_r$, i.e. the following commutes:
\[\begin{tikzcd}
C_r(G) \ar[d,"d_r"] \ar[r,"\Phi_r"] & \tilde C_r(G) \ar[d,"\tilde d_r"] \\
C_{r-1}(G) \ar[r,"\Phi_{r-1}"] & \tilde C_{r-1}(G)
\end{tikzcd}\]
\end{proposition}
\begin{proof}
An element of $C_r(G)$ is of the form
$$\sum_{k=1}^m \left( \sum_{g \in G} n_{k,g} g \right) \cdot (g_{k,1}, \cdots g_{k,r})$$
Firstly $\Phi_r$ is a $G$-module morphism because
$$\Phi_r(h \cdot (g \cdot (g_1, \cdots g_r))) = \Phi_r (hg \cdot (g_1, \cdots g_r)) = (hg, \cdots (hgg_1, \cdots g_r)) = h \cdot \Phi_r( g \cdot (g_1, \cdots g_r))$$
and $\Phi_r$ respects addition by definition. 

$\Phi_r$ is an injection because
\begin{align*}
& \Phi_r \left( \sum_{k,g} n_{k,g} g \cdot (g_{k,1}, \cdots g_{k,r}) \right) = 0 \\
\iff &  \sum_{k,g} n_{k,g} (g, g g_{k,1}, \cdots (g g_{k,1} \cdots g_{k,r})) = 0 \\
\iff & \forall k, g, n_{k,g} = 0
\end{align*}
where we observe that $(g, gg_{k,1}, \cdots (gg_{k,1} \cdots g_{k,r})) = (h, hh_{k,1}, \cdots (hh_{k,1} \cdots h_{k,r})) \iff g=h, g_{k,j} = h_{k,j}$ by term-by-term comparison. 

$\Phi_r$ is a surjection because
$$(h_0, \cdots h_r) = \Phi_r(h_0 \cdot (h_0^{-1} h_1, h_1^{-1} h_2, \cdots h_{r-1}^{-1} h_r) ) $$

Lastly $\Phi \circ d = \tilde d \circ \Phi$ because
\begin{align*}
& (\Phi_r \circ d_r) (g \cdot (g_1, \cdots g_r)) \\
=& \Phi_r \left( gg_1 \cdot (g_2, \cdots g_r) + \left( \sum_{k=1}^{r-1} (-1)^k g \cdot (g_1, \cdots g_k g_{k+1}, \cdots g_r) \right) + (-1)^r g \cdot (g_1, \cdots g_{r-1}) \right) \\
=& (\tilde g_1, \tilde g_2, \cdots \tilde g_r) + \left( \sum_{k=1}^{r-1} (-1)^k (\tilde g_0, \tilde g_1, \cdots \tilde g_{k-1}, \tilde g_{k+1}, \cdots \tilde g_r) \right) + (-1)^r (\tilde g_0, \cdots \tilde g_{r-1}) \\
=& \tilde d_r (\tilde g_0, \tilde g_1, \cdots \tilde g_r) \\
=& (\tilde d_r \circ \Phi_{r-1}) (g \cdot (g_1, \cdots g_r))
\end{align*}
where $\tilde g_k = gg_1 \cdots g_k$. 
\end{proof}

\begin{proposition}
$C_\bullet(G) \rightarrow \ZZ \ra 0$ and $\tilde C_\bullet(G) \rightarrow \ZZ \ra 0$ are exact sequences.
\end{proposition}
\begin{proof}
Fixing an element $h$, define maps $f_r : C_r(G) \ra C_{r+1}(G)$ by
$$f_r : \varphi \mapsto ((g_0, \cdots g_r) \mapsto (h, g_0, \cdots g_r))$$
Then it can be verified that $\d_{r+1} \circ f_r + f_{r-1} \circ \d_r = 1$, which tells us that whenever $\d_r x = 0$, $x = \d_{r+1} (f_r(x))$. The result for $\tilde C_r(G)$ follows since $C_r(G) \cong \tilde C_r(G)$.
\end{proof}

This gives our standard resolutions for group cohomology and homology.

We will frequently use the inhomogeneous cochains to describe group cohomology. This is a map $\varphi: \bigoplus_{\gamma \in G^r} \ZZ[G] \cdot \gamma \ra M$ but since this map is determined uniquely by its values on $\gamma$, a $r$-cochain can simply be regarded as a $\textbf{Set}$ map $G^r \ra M$. By slight abuse of notation we will refer to this map as a cochain. A cocycle must satisfy an additional condition:
$$0 = \d \varphi(g_0, \cdots g_n) = g_0 \varphi(g_1, \cdots g_n) + \sum_{k=1}^{n-1} \varphi(g_1, \cdots g_k g_{k+1}, \cdots g_n) + (-1)^n \varphi(g_0, \cdots g_{n-1})$$
In particular, we call a 1-cocycle a \emph{crossed homomorphism}, as it satisfies
$$\varphi(gh) = g\varphi(h) + \varphi(g)$$
behaving like a homomorphism, but not quite.

\begin{proposition}
Given an exact sequence of $R$-modules
$$0 \ra A \ra B \ra C \ra 0$$
the connecting homomorphisms of group cohomology and homology are given by
\begin{align*}
\delta : & H^n(G,C) \ra H^{n+1}(G,A) \\
& [\varphi] \mapsto [\d \tilde \varphi]
\end{align*}
where $\tilde \varphi : P_n \rightarrow B$ is a lift of $\varphi : P_n \rightarrow C$ and
\begin{align*}
\delta : & H_n(G,C) \ra H_{n-1}(G,A) \\
& \left[ \sum_p p\otimes a_p \right] \mapsto \left[ \sum_p \d p \otimes \tilde a_p \right]
\end{align*}
where $\tilde a_p$ is a lift of $a_p$ through surjection $B \rightarrow C$.
\end{proposition}
\begin{proof}
This is a direct application of Proposition~\ref{ext tor delta}.
\end{proof}

Now we define Tate cohomology, a concatenation of group cohomology and homology. For reasons soon to be apparent, consider the following definition:
\begin{definition}
For a finite group $G$, Tate cohomology groups are defined as:
\begin{align*}
H_T^r(G,M) = \begin{cases} H^r(G,M) & \text{ if }r>0 \\ \on{coker} \widehat{\Nm}_G = \frac{M^G}{\Nm_G(M)} & \text{ if }r=0 \\ \on{ker} \widehat{\Nm}_G = \frac{\ker \Nm_G}{I_G M}  & \text{ if }r=-1 \\ H_{-(r+1)}(G,M) & \text{ if }r<-1 \end{cases}
\end{align*}
where $\on{Nm}_G : M \rightarrow M$ is the map $a \mapsto \sum_{g\in G} g \cdot a$ and $\widehat{\Nm_G} : M_G \rightarrow M^G$ is the induced map. 
\end{definition}

Then we can splice group homology long exact sequence in group cohomology long exact sequence induced by
$$0 \ra A \ra B \ra C \ra 0$$
by applying snake lemma to the following diagram:
\[\begin{tikzcd}
 & & H_T^{-1}(G,A)\ar[r]\ar[d] & H_T^{-1}(G,B)\ar[r]\ar[d] & H_T^{-1}(G,C)\ar[d] & & \\
\cdots \ar[r] & H_1(G,C) \ar[ru,dashed]\ar[r] & H_0(G,A)\ar[r]\ar[d,"\widehat{\Nm}_G"] & H_0(G,B) \ar[r]\ar[d,"\widehat{\Nm}_G"] & H_0(G,C)\ar[d,"\widehat{\Nm}_G"] \ar[r] & 0 & \\
& 0 \ar[r] & H^0(G,A)\ar[r]\ar[d] & H^0 (G,B)\ar[r]\ar[d] & H^0 (G,C)\ar[r]\ar[d] & H^1(G,C) \ar[r] & \cdots \\
 & & H_T^{0}(G,A) \ar[r] & H_T^{0}(G,B) \ar[r] & H_T^{0}(G,C) \ar[ru,dashed] & &
\end{tikzcd}\]
We must check that the induced maps (denoted using dashed arrows) are well-defined. For convenience in explaining diagram chase, let's refer to maps using directions. For example, we'll refer to the composition $H_1(G,C) \rightarrow H_0(G,A) \rightarrow H^0(G,A) \ra H^0(G,B)$ as the right-down-right map starting at $H_1(G,C)$. 

($H_1(G,C) \ra H_T^{-1}(G,A)$) The right-down-right map starting at $H_1(G,C)$ is equal to right-right-down map, which is zero due to exactness of the rows. Due to injectivity of $H^0(G,A) \rightarrow H^0(G,B)$, this means that right-down map starting at $H_1(G,C)$ is zero, i.e. image of $H_1(G,C) \ra H_0(G,A)$ is in kernel of $\Nm_G$. Thus the induced map $H_1(G,C) \ra H_T^{-1}(G,A)$ is well-defined.

($H_T^0(G,C) \ra H^{1}(G,A)$) The right-down-right map starting at $H_0(G,B)$ is equal to down-right-right map, which is zero due to exactness of the rows. Thus the down-right map starting at $H_0(G,C)$ is always zero, as we can always lift through the surjection $H_0(G,B) \rightarrow H_0(G,C)$. Thus $\Nm_G(H_0(G,C)) \subseteq \ker(H^0(G,C) \rightarrow H^1(G,C))$ and therefore we get an induced map $H_T^0(G,C) \ra H^1(G,C)$.

$H_T^{-1}(G,C) \rightarrow H_T^0(G,A)$ is given by snake lemma:
\begin{align*}
\delta : & H_T^{-1}(G,C) \rightarrow H_T^0(G,A) \\
& [a] \mapsto [\Nm_G(\tilde a)] \text{ where $\tilde a \in B$ is a lift}
\end{align*}
Therefore we get a long exact sequence
\begin{align*}
\cdots & \ra H_T^r(G,A) \ra H_T^r(G,B) \ra H_T^r(G,C) \ra \\
& \ra  H_T^{r+1}(G,A) \ra H_T^{r+1}(G,B) \ra H_T^{r+1}(G,C) \ra \cdots
\end{align*}
that is exact for all $r \in \ZZ$.

We also introduce the notion of cup product for Tate cohomology. We will later use it to describe inverse of the local reciprocity map as a cup product with a distinguished cohomology class.

\begin{proposition}
There is a unique covariant biadditive pairing $\smile$:
\begin{align*}
& H_T^r(G,M) \times H_T^s(G,N) \ra H_T^{r+s}(G,M\otimes N) \\
& (a, b) \mapsto a \smile b
\end{align*}
such that $\smile$ is a natural transformation between bifunctors $H_T^r(G,-) \times H_T^r(G,-)$ and $H_T^{r+s}(G,-)$, and the following properties hold:
\begin{enumerate}
\item For $r=s=0$, the map is given by $[a] \smile [b] = [a \otimes b]$
\item Suppose we have exact sequence $0 \ra L \ra M \ra N \ra 0$ such that
$$0 \ra L \otimes A \ra M \otimes A \ra N \otimes A \ra 0$$
is exact. Then
$$(\delta a) \smile b = \delta(a \smile b), a \in H^r(G,N), b \in H^s(G,A)$$
\item Suppose we have exact sequence $0 \ra L \ra M \ra N \ra 0$ such that
$$0 \ra A \otimes L \ra A \otimes M \ra A \otimes N \ra 0$$
is exact. Then
$$a \smile \delta b = (-1)^r \delta(a \smile b), a \in H^r(G,A), b \in H^s(G,N)$$
\end{enumerate}
\end{proposition}

We will give some details to the computation of cup product later.

\subsection{Restriction, Inflation, Corestriction}

Here we introduce three important maps that relate cohomology of a module to its submodule and quotient modules. Firstly consider a general situation:
\begin{proposition}\label{cor res general}
Suppose $M$ is a $G$-module and $M'$ is a $G'$-module with $\alpha : G' \ra G$ and $\beta : M \rightarrow M'$ that satisfy a compatibility condition: $\beta(\alpha(g') \cdot m) = g' \cdot \beta(m)$. Then we have an induced map:
\begin{align*}
H_T^r(G,M) \rightarrow H_T^r(G',M')
\end{align*}
\end{proposition}
\begin{proof}
For $r>0$ and $\varphi \in C^r(G,M)$, consider the map $\varphi \mapsto \beta \circ \varphi \circ \alpha^n$. The compatibility condition is required for proving that this map commutes with coboundary map $\d$. Once this is done we indeed get an induced map of Tate cohomology for $r>0$.

Now the induced map for the rest of the Tate cohomology groups are obtained by dimension shifts; consider the cohomology long exact sequence associated to
\begin{align*}
& 0 \ra I_G \otimes M \ra \ZZ[G] \otimes M \ra M \ra 0 \\
\implies & H_T^r(G,M) \cong H_T^{r+1}(G,I_G \otimes M)
\end{align*}
and thus we get induced maps $H_T^r(G,M) \rightarrow H_T^r(G',M')$ for all $r \in \ZZ$.
\end{proof}

From this result we define the following maps:
\begin{align*}
& \on{Res} : H_T^r(G,M) \ra H_T^r(H,M) \text{ if $H \le G$} \\
& \on{Inf} : H_T^r(G/H,M^H) \ra H_T^r(G,M) \text{ if $H \trianglelefteq G$} \\
& \on{Cor} : H_T^r(H,M) \ra H_T^r(G,M) \text{ if $[G:H] < \infty$}
\end{align*}
Here $\on{Res}$ literally restricts a cocycle to its subgroup and $\on{Inf}$ simply composes with the canonical projection. The only tricky definition is the $\on{Cor}$ map, which is defined as a composition of isomorphism $H^r(H,M) \cong H^r(G,\on{Ind}_H^G M)$ in Shapiro's lemma and the map obtained by applying Proposition~\ref{cor res general} to the following map:
\begin{align*}
& \Hom_{\ZZ[H]}(\ZZ[G] , M) \rightarrow M \\
& f \mapsto \sum_{k=1}^n g_k \cdot f(g_k^{-1})
\end{align*}
which composes into
\begin{align*}
\on{Cor} : & H^r(H,M) \rightarrow H^r(G,M) \\
& [f] \mapsto \left[ p \mapsto \sum_{k=1}^n g_k f(g_k^{-1} p) \right]
\end{align*}
where $f \in \Hom_{\ZZ[H]}(P_r,A)$. As before, the definition of corestriction map extends to all Tate cohomology by dimension shift.

\begin{proposition}
For a $G$-module $M$ and a finite-index subgroup $H \le G$,
\begin{align*}
\on{Cor} \circ \on{Res} = [G:H]
\end{align*}
\end{proposition}
\begin{proof}
For $\varphi \in \Hom_{\ZZ[G]}(P_n, M)$, 
\begin{align*}
\on{Cor}(\on{Res}([\varphi])) = \left[ p \mapsto \sum_{k=1}^n g_k \varphi(g_k^{-1} p) \right] = \left[ p \mapsto \sum_{k=1}^n g_k g_k^{-1} \varphi(p) \right] = [G:H] [\varphi]
\end{align*}
\end{proof}

\begin{proposition}[Inflation-Restriction Sequence]
Suppose for a normal subgroup $H \trianglelefteq G$ and integer $r>0$, whenever $0<i<r, H^r(H,M)=0$. Then the following sequence is exact:
\begin{align*}
0 \ra H_T^r(G/H, M^H) \xrightarrow{\text{Inf}} H_T^r(G, M) \xrightarrow{\text{Res}} H_T^r(H, M)
\end{align*}
\end{proposition}
\begin{proof}
We prove the $r=1$ case and then induct by dimension-shifting. For $r=1$, suppose $\varphi \in Z^1(G,M)$ satisfies $\forall h \in H, \varphi(h) = hm - m$. Let $\varphi_0= \varphi - (gm - m)$, which satisfies $[\varphi] = [\varphi_0]$. Then $\forall h \in H, \varphi_0(h)=0$ and thus it induces $\hat \varphi_0: G/H \rightarrow M^H$, whose inflation is $\varphi_0$. This establishes exactness at $H^1(G,M)$. For exactness at $H^r(G/H,M^H)$, suppose $\varphi: G/H \rightarrow M^H$ satisfies $\on{Inf}(\varphi): g \mapsto gm-m$ with $m \in M^H$ because $\on{Inf}(\varphi) : h \mapsto 0$. Then $\varphi([g]) = \on{Inf}(\varphi) = gm - m = [g] m - m$ as desired.

The inductive step is dimension shift; suppose the hypothesis is true for $r-1$ with $r>1$. Then consider exact sequence
\begin{align*}
& 0 \ra M \ra \ZZ[G] \otimes M \ra J_G \otimes M \ra 0 \\
\implies & H_T^r(G,M) \cong H_T^{r-1}(G,J_G \otimes M)
\end{align*}
from which we get
\begin{align*}
& 0 \ra H_T^{r-1} (G/H, (J_G \otimes M)^H) \xrightarrow{\text{Inf}} H_T^{r-1} (G, J_G \otimes M) \xrightarrow{\text{Res}} H_T^{r-1} (H, J_G \otimes M) \text{ is exact} \\
\implies & 0 \ra H_T^r(G/H, M^H) \xrightarrow{\text{Inf}} H_T^r(G, M) \xrightarrow{\text{Res}} H_T^r(H, M) \text{ is exact}
\end{align*}
and similarly we can use exact sequence
$$0 \ra I_G \otimes M \ra \ZZ[G] \otimes M \ra M \ra 0$$
to dimension-shift from $r=1$ to $r=0, -1, \cdots$.
\end{proof}

\subsection{Cohomology of Profinite Groups}

We define cohomology of profinite groups (infinite Galois groups) too. A profinite group is a Hausdorff, compact, totally disconnected group. We restrict our scope to discrete module over a profinite group, which means to say that group action $G \times M \ra M$ is a continuous map when $M$ is given the discrete topology. This is equivalent to saying that $M = \bigcup M^H$ ($H$ running over open normal subgroups) or stabilizer of any element in $M$ is an open subgroup of $G$.

The category of discrete $G$-modules is a category with enough injectives that we can define cohomology groups by taking injective resolutions. The cohomology groups must be calculated using continuous cocycles and coboundaries. 

The following proposition relates cohomology of a profinite group to its finite quotients, and it will be useful for us later when we relate cohomology of an infinite Galois extension to cohomology of finite Galois extension.
\begin{proposition}\label{profinite limit}
For $M$ a discrete module over a profinite group $G$, we have an isomorphism:
\begin{align*}
H^r(G,M) \cong \lim_\rightarrow H^r(G/H,M^H)
\end{align*}
where $H$ are open normal subgroups and $H^r(G/H,M^H)$ are connected by inflation maps.
\end{proposition}
\begin{proof}
Suppose $\varphi: G^r \rightarrow M$ is a continuous cochain. $\varphi(G^r)$ is compact because $G^r$ is compact. $\varphi(G^r)$ is discrete because $M$ is discrete. Thus $\varphi(G^r)$ is a finite set, and is contained in some $M^{H_0}$ for some open normal subgroup $H_0$ because $M=\bigcup M^{H}$ and intersection of normal subgroups is normal. For each $m \in M$, $\varphi^{-1}(m)$ is open and it contains a translate of $H(m)^r$ for some open normal subgroup $H(m)$. $\varphi$ then factors through $G / \left(\bigcap_{m \in \varphi(G^r)} H(m)\right)$, an open normal subgroup. Letting $H = H_0 \cap \bigcap_{m \in \varphi(G^r)} H(m)$, we see that $\varphi$ factors through some $\tilde \varphi : (G/H)^r \rightarrow M^H$. Thus every continuous cochain arises as an inflation of some map $(G/H)^r \ra M^H$. This gives a surjection
$$\lim_\rightarrow C^r( G/H, M^H) \ra C^r(G,M)$$
This map is injection too. As taking limit is compatible with taking cohomology of a complex, we have then
$$H^r(G,M) \cong \lim_\rightarrow H^r(G/H, M^H)$$
\end{proof}

\subsection{Cohomology Computations}

We discuss several computations of cohomology. Some cohomology of induced modules, coinduced modules, number field (as a Galois module), $\ZZ, \QQ, \QQ/\ZZ$ can be computed.
\begin{proposition}[Shapiro's Lemma]
If $H \le G$ is a subgroup,
\begin{align*}
H^r(H,M) &\cong H^r(G, \Hom_{\ZZ[H]}(\ZZ[G],M) ) \\
H_r(H,M) &\cong H_r(G, \ZZ[G] \otimes_{\ZZ[H]} M )
\end{align*}
\end{proposition}
\begin{proof}
Suppose $P_\bullet \ra \ZZ \ra 0$ is a projective resolution of $\ZZ[G]$-modules. As $\ZZ[G]$ is a free $\ZZ[H]$-module, this is also a projective resolution of $\ZZ[H]$-modules. Apply now the functors $\Hom_{\ZZ[H]}(-,M)$ and $\Hom_{\ZZ[G]}(-,\Hom_{\ZZ[H]}(\ZZ[G],M)))$, and apply adjoint associativity (Proposition~\ref{adjoint associativity}):
\begin{align*}
\Hom_{\ZZ[G]}(P_n, \Hom_{\ZZ[H]}(\ZZ[G],M)) \cong \Hom_{\ZZ[H]}(\ZZ[G] \otimes_{\ZZ[G]} P_n, M) \cong \Hom_{\ZZ[H]}(P_n, M)
\end{align*}
which gives isomorphism of two chain complexes, which implies that their cohomology groups, i.e. $H^r(H,M)$ and $H^r(G,\Hom_{\ZZ[H]}(\ZZ[G],M))$ are isomorphic.

We can also apply functors $\otimes_{\ZZ[G]}(\ZZ[G] \otimes_{\ZZ[H]} M)$ and $\otimes_{\ZZ[H]} M$ to $P_\bullet \ra \ZZ \ra 0$ to get chain complexes that are quite obviously equal because 
$$A \otimes_{\ZZ[G]}(\ZZ[G] \otimes_{\ZZ[H]} M) \cong (A \otimes_{\ZZ[G]}\ZZ[G]) \otimes_{\ZZ[H]} M \cong A \otimes_{\ZZ[H]} M$$
implying that $H_r(H,M) \cong H_r(G,\ZZ[G] \otimes_{\ZZ[H]} M)$ too.
\end{proof}

\begin{corollary}
\begin{align*}
& H^r(G,\Hom_\ZZ(\ZZ[G],M)) = 0 \\
& H_r(G,\ZZ[G]\otimes_\ZZ M) = 0
\end{align*}
\end{corollary}

\begin{proposition}
For a $G$-module $M$,
$$H_1(G,\ZZ) \cong I_G / I_G^2 \cong G^{\on{ab}} = G / \langle\{(gh)/(hg) | g, h \in G\} \rangle$$
\end{proposition}
\begin{proof}
The first isomorphism holds by dimension-shift on the exact sequence:
$$0 \ra I_G \ra \ZZ[G] \ra \ZZ \ra 0$$
($H_0(G,\ZZ[G]) \ra H_0(G,\ZZ)$ is a zero map) 

The second isomorphism can be obtained by the mapping
\begin{align*}
& G \rightarrow I_G / I_G^2 \\
& g \mapsto [g - 1]
\end{align*}
The map is a homomorphism because $[gh -1] = [(g-1)(h-1) + (g-1) + (h-1)] = [g-1] + [h-1]$.

As $ghg^{-1}h^{-1} - 1 = (gh-1)(g^{-1}h^{-1}-1)+(g-1)(h-1)+(g^{-1}-1)(h^{-1}-1)-(g-1)(g^{-1}-1)-(h-1)(h^{-1}-1)$, we see that this factors through $G^{\on{ab}}$.

Let's construct inverse by mapping $g-1 \in I_G$ to $[g] \in G^{\on{ab}}$. As $(g-1)(h-1) = (gh-1)+(g-1)+(h-1)$, the map factors through $I_G^2$, and it can be checked directly that composition of this map with the earlier map is identity.
\end{proof}

In our context we will frequently consider $H^r(G,M)$ when $G = \Gal(L/K)$ and $M= L$ or $L^\times$ where $L,K$ are number fields. Therefore the following are fundamental:
\begin{proposition}[Hilbert's Theorem 90]
For a (possibly infinite) Galois extension $L/K$ and any closed subgroup $H \le \Gal(L/K)$,
$$H^1(H,L^\times)=0$$
\end{proposition}
\begin{proof}
Firstly assume $L/K$ is finite Galois. We want to show that crossed homomorphism $\varphi: H \ra L^\times$ is principal. By linear independence of automorphisms, let $b = \sum_{\sigma \in H} \varphi(\sigma) \sigma (a)$ for some $a \in L^\times$. Then for $\tau \in H$, $\varphi(\tau) = \tau(b^{-1}) / b^{-1}$ since
\begin{align*}
\tau b &= \sum_{\sigma \in H} (\tau \cdot \varphi(\sigma)) \tau \sigma(a) = \sum_{\sigma \in H} \varphi(\tau)^{-1} \varphi(\tau \sigma) \tau \sigma(a) = \varphi(\tau)^{-1} \sum_{\sigma \in H} \varphi(\tau \sigma) \tau \sigma(a) = \varphi(\tau)^{-1} b
\end{align*}

Now consider a general Galois extension $L/K$. We may assume $H=\Gal(L/K)$ since a closed subgroup is of the form $\Gal(L/K')$, where $L/K'$ is also Galois. The proof completes by:
$$H^1( \Gal(L/K), L^\times) \cong \lim_\rightarrow H^1(\Gal(K'/K), (K')^\times) = 0$$
\end{proof}

\begin{proposition}[Normal Basis Theorem]
For a finite Galois extension $L/K$ and $\forall r>0$,
$$H^r(\Gal(L/K),L)=0$$
\end{proposition}
\begin{proof}
Classical normal basis theorem states that there is $\alpha \in L$ such that its Galois orbit is a basis of the extension. This translates into existence of isomorphism of $G$-modules $L \cong K[G]$
\begin{align*}
\sum_{k=1}^n a_k \sigma_k(\alpha) \mapsto \sum_{k=1}^n a_k \sigma_k
\end{align*}
Therefore cohomologically we see that
\begin{align*}
H^r(G,L) \cong H^r(G, K[G]) \cong H^r(G, \ZZ[G] \otimes_\ZZ K) \cong H^r(G, \Hom_{\ZZ}(\ZZ[G],K)) \cong H^r(1,K) = 0
\end{align*}
\end{proof}

\begin{proposition}\label{integer cohomology}
For a finite group $G$ acting on $\QQ, \ZZ, \QQ/\ZZ$ trivially,
\begin{align*}
\forall r, & H_T^r(G,\QQ)=0 \\
& H_T^0(G,\ZZ) \cong \ZZ/|G| \\
& H_T^1(G,\ZZ)=0 \\
& H_T^2(G, \ZZ) \cong \Hom(G,\QQ/\ZZ)
\end{align*}
\end{proposition}
\begin{proof}
Multiplication by $|G|$ on $H_T^r(G,\QQ)$ factors through Res and Cor, and it's isomorphism (inverse is $\frac1{|G|}$). But through Res, we get a zero map. Since zero map is isomorphism, $H_T^r(G,\QQ)=0$.

Since $G$-action on $\ZZ$ is trivial, a 1-cocycle (crossed homomorphism) is a group homomorphism, but since $G$ is torsion, there is no nontrivial map to $\ZZ$. Thus $H^1(G,\ZZ)=0$. $H^0_T(G,\ZZ) = \ZZ / |G|$ follows by definition, considering that $G$-action is trivial. Lastly the exact sequence
$$0 \ra \ZZ \ra \QQ \ra \QQ / \ZZ \ra 0 $$
gives $H^2_T(G,\ZZ) \cong H_T^1(G,\QQ/\ZZ) \cong \Hom(G,\QQ/\ZZ)$ by connecting homomorphism and noting that $G$ acts on $\QQ/\ZZ$ trivially.
\end{proof}

Interestingly, Tate cohomology is 2-periodic for cyclic $G$. Thus only two of the cohomology groups matter, and we can develop some neat results about the ratio $|H_T^1| / |H_T^0|$.

\begin{proposition}
For $M$, a module over finite cyclic group $G$, we have an isomorphism
$$H_T^r(G,M) \cong H_T^{r+2}(G,M) $$
\end{proposition}
\begin{proof}
Suppose $G = \langle \sigma \rangle$. The following sequence is exact:
$$0 \rightarrow \ZZ \xrightarrow{\mu} \ZZ[G] \xrightarrow{\sigma - 1} \ZZ[G] \xrightarrow{\epsilon} \ZZ \ra 0$$
and as $\ZZ$ and $I_G = \ker(\ZZ[G] \ra \ZZ)$ are free abelian groups, their first Tor groups vanish and thus by applying $\otimes M$, we get an exact sequence
$$0 \ra M \ra \ZZ[G] \otimes M \ra \ZZ[G] \otimes M \ra M \ra 0$$
The middle two modules have vanishing cohomology and thus we can apply dimension-shift twice to obtain
$$H_T^r(G,M) \cong H_T^{r+2}(G,M)$$
\end{proof}

\begin{definition}
The Herbrand quotient of a module $M$ over finite cyclic group $G$ is defined as
\begin{align*}
h(M) := \frac{|H_T^0(G,M)|}{|H_T^1(G,M)|} (= \frac{H_T^{2n}(G,M)}{H_T^{2n+1}(G,M))})
\end{align*}
whenever the cardinalities considered are finite. 
\end{definition}

Herbrand quotient can be thought as an analogue of the Euler characteristic.

\begin{proposition}
Given exact sequence $0 \rightarrow M' \rightarrow M \rightarrow M'' \rightarrow 0$, if any of the two Herbrand quotients are defined, then so is the third and the following equality holds:
$$h(M) = h(M') h(M'') $$
\end{proposition}
\begin{proof}
Truncate the cohomology long exact sequence as follows:
\begin{align*}
0 \ra K \ra H_T^0(M') \ra H_T^0(M) \ra H_T^0 (M'') \ra H_T^1(M') \ra H_T^1(M) \ra H_T^1(M'') \ra K' \ra 0
\end{align*}
where $K = \on{coker}(H_T^{-1} M \ra H_T^{-1}(M''))$ and $K' = \on{coker}(H_T^1(M) \ra H_T^1(M''))$ (the two are isomorphic). 

This immediately tells us that when two Herbrand quotients are defined (i.e. when cohomologies are finite) then the third is also defined, because cohomologies are sandwiched between each other. The identity $h(M) = h(M') h(M'')$ follows from the following computation:
\begin{align*}
& 0 \ra A_0 \ra A_1 \ra \cdots \ra A_r \ra 0 \text{ exact} \\
\implies & \frac{|A_0| \cdot |A_2| \cdots}{|A_1| \cdot |A_3| \cdots} = 1
\end{align*}
(when $A_j$ are finite) which follows by splitting the exact sequence into short exact sequences.
\end{proof}

\begin{proposition}
If $M$ is a finite $G$-module, then $h(M)=1$. 
\end{proposition}
\begin{proof}
Write $G = \langle \sigma \rangle$. Then
\begin{align*}
& 0 \ra M^G \ra M \xrightarrow{\sigma - 1} M \ra M_G \ra 0 \\
& 0 \ra H_T^{-1}(G, M) \ra M_G \xrightarrow{\on{Nm}_G} M^G \ra H_T^0(G,M) \ra 0
\end{align*}
Now apply the $\frac{|A_0| \cdot |A_2| \cdots}{|A_1| \cdot |A_3| \cdots} = 1$ to each exact sequence to obtain $|H_T^{-1}(G,M)| = |H_T^0(G,M)|$.
\end{proof}

\begin{corollary}
Let $\alpha$ be a $G$-module map with finite kernel and cokernel. If either $h(M)$ or $h(N)$ is defined, then so is another, and $h(M) = h(N)$.
\end{corollary}
\begin{proof}
Apply previous two propositions to the following exact sequences:
\begin{align*}
& 0 \ra \alpha(M) \ra N \ra \on{coker}(\alpha) \ra 0 \\
& 0 \ra \on{ker}(\alpha) \ra M \ra \alpha(M) \ra 0
\end{align*}
\end{proof}

\pagebreak
\section{Local Reciprocity Law}

Denote $G = \Gal(L/K)$ where $L/K$ is a finite Galois extension. We also denote 
$$H^r(L/K):=H^r(\Gal(L/K),L^\times)$$

In this section, we state and prove the local reciprocity law.

\subsection{Statement}

Local reciprocity law is an isomorphism between the Galois group of an abelian extension $L/K$ of local fields and $K^\times$ modulo norm subgroup $\Nm_{L/K}(L^\times)$. It allows a characterization of extensions of $K$ using information internal to $K$.

\begin{theorem}[Local Reciprocity Law]
For a local field $K$, there is a homomorphism
$$\phi_K: K^\times \rightarrow \Gal(K^{\on{ab}}/K)$$
such that
\begin{enumerate}
\item For uniformizer $\pi \in K^\times$ and a finite unramified extension $L/K$, $\phi_K(\pi)$ acts on $L$ as $\on{Frob}_{L/K}$.
\item For a finite abelian extension $L/K$, by composing $\phi$ with restriction to $\Gal(L/K)$ we get a surjection whose kernel is $\Nm_{L/K}(L^\times)$ an induced isomorphism
$$\phi_{L/K} : \frac{K^\times}{\Nm_{L/K}(L^\times)} \cong \Gal(L/K)$$
\end{enumerate}
\end{theorem}

We prove this theorem by establishing $\frac{K^\times}{\Nm_{L/K}(L^\times)} \cong \Gal(L/K)^{\on{ab}}$ first, and then using a compatibility condition to get the map $\phi_K$. The relation $\frac{K^\times}{\Nm_{L/K}(L^\times)} \cong \Gal(L/K)^{\on{ab}}$ is in turn proven as a cohomological result:
$$\Gal(L/K)^{\on{ab}} \cong H_T^{-2}(G,\ZZ) \cong H_T^0(G,L^\times) = \frac{K^\times}{\Nm_{L/K}(L^\times)}$$
using dimension shifts. We also prove that $\phi_{L/K}^{-1}$ is given by:
$$\phi_{L/K}^{-1} : [\sigma] \mapsto \left[ \prod_{\tau \in \Gal(L/K)} \varphi(\tau,\sigma) \right]$$
where $[\varphi]$ is a generator of $H^2(L/K)$. This map is a boundary map induced by a certain short exact sequence of $G$-modules. 

We briefly summarize the proof. This summary can be read before and after studying the proof itself to fortify one's understanding.
\begin{enumerate}
    \item We first prove Tate's theorem, which is an an abstract veresion of the local reciprocity law. Tate's theorem relates $r$-th cohomology group and $(r+2)$-th cohomology group, and this is proven via two dimension shifts. The dimension shifts are done using the modules $\ZZ[G]$ and $L^\times (\varphi)$, whose cohomology groups vanish. But the vanishing is shown to hold subject to the condition that $H^2(L/K)\cong \frac1{[L:K]}\ZZ / \ZZ$, which is the task to establish in next two steps.
    \item We show that $H^2(L/K) \cong \frac1{[L:K]}\ZZ/\ZZ$ holds for a finite unramified extension $L/K$ by changing the coefficient module from $L^\times$ to $\ZZ$, and then to $\QQ/\ZZ$ using exact sequences
    \begin{align*}
        & 0 \ra U_L \ra L^\times \xrightarrow{\on{ord}_L} \ZZ \ra 0\\
        & 0 \ra \ZZ \ra \QQ \ra \QQ/\ZZ \ra 0
    \end{align*}
    Thus we demand that $H_T^r(G,U_L)=0$, and the nontrivial task is to show this for $r=0$, i.e. that $\Nm_{L/K}$ is surjective. Construction of preimage in surjectivity proof is done by a limit process, using quotient groups of the filtration $U_L \supset U_L^{(2)} \supset U_L^{(3)} \supset \cdots$ which are $\ell^+$ or $\ell^\times$. Norm maps $\ell^\times \rightarrow k^\times$ and $\ell^+ \rightarrow k^+$ are shown to be surjective and by limit process we obtain the desired surjectivity relation.
    \item We show that $H^2(L/K) \cong \frac1{[L:K]}\ZZ/\ZZ$ holds for a general finite Galois extension by embedding into $H^2(\bar K / K)$. More precisely, we consider the following diagram:
    \[\begin{tikzcd}
        0 \ar[r] & \frac1{[L:K]}\ZZ/\ZZ \ar[r] \ar[d] & H^2(K^{\on{ur}}/K) \ar[r,"\on{Res}"] \ar[d,"\on{Inf}"] & H^2(L^{\on{ur}} / L) \ar[d,"\on{Inf}"] \\
        0 \ar[r] & H^2(L/K) \ar[r] & H^2(\bar K / K) \ar[r,"\on{Res}"] & H^2(\bar L / L)
    \end{tikzcd}\]
    Groups of the map $H^2(K^{\on{ur}}/K) \rightarrow H^2(L^{\on{ur}}/L)$ embed into groups of the map $H^2(\bar K / K) \xrightarrow{\on{Inf}} H^2(\bar L/L)$, whose kernel is isomorphic to $\frac1{[L:K]}\ZZ/\ZZ$, which embeds to $H^2(L/K)$. Thus it suffices to show $|H^2(L/K)|\le n$, which we prove by induction and solvability; we find a cyclic subextension $K'/K$ and then
    $$0 \ra H^2(K'/K) \xrightarrow{\on{Inf}} H^2(L/K) \xrightarrow{\on{Res}} H^2(L/K')$$
    and we get $|H^2(L/K)|\le |H^2(K'/K)| \cdot |H^2(L/K')| \le n$.
\end{enumerate}

\subsection{Proof Part I: Tate's Theorem}

Tate's theorem is a generalized cohomological version of the local reciprocity law. In this section we prove Tate's theorem, and in later sections we prove a condition that allows us to apply Tate's theorem to our situation.

\begin{theorem}[Tate]\label{tate}
For a finite group $G$ and a $G$-module $C$, suppose that $\forall H \le G$, $H^1(H,C)=0$ and $H^2(H,C) \cong \ZZ/|H|$. Then
$$H_T^r(G,\ZZ) \cong H_T^{r+2}(G, C)$$
\end{theorem}
\begin{proof}
The key idea is to use dimension shifts twice, and to do that we need to construct an appropriate module with vanishing cohomology. Letting $[\varphi]$ be a generator of $H^2(G,C) \cong \ZZ/|G|$, define the \emph{splitting module}
\begin{align*}
& C(\varphi) = C \oplus \bigoplus_{\sigma \in G, \sigma \neq 1} \ZZ \cdot x_\sigma \\
& \sigma \cdot x_\tau := x_{\sigma\tau} - x_\sigma + \varphi(\sigma,\tau)
\end{align*}
(we use formal symbol $x_1=\varphi(1,1)=1$ here) Let's verify that this is a $G$-module action:
\begin{align*}
(\rho \sigma) \cdot x_\tau &= x_{\rho \sigma \tau} - x_{\rho \sigma} + \varphi(\rho\sigma, \tau) \\
\rho \cdot (\sigma \cdot x_\tau) &= \rho(x_{\sigma \tau} - x_\sigma + \varphi(\sigma, \tau)) \\
&= x_{\rho \sigma \tau} - x_\rho + \varphi(\rho, \sigma \tau) - (x_{\rho \sigma} - x_\rho + \varphi(\rho, \sigma)) + \rho \varphi(\sigma, \tau)
\end{align*}
and these are equal by the 2-cocycle condition:
$$\rho \varphi(\sigma, \tau) + \varphi(\rho, \sigma \tau) = \varphi(\rho \sigma, \tau) +\varphi(\rho, \sigma)$$

We also define a map $\alpha: C(\varphi) \rightarrow I_G$ given by $\alpha(c) = 0$ for each $c \in C$ and $\alpha(x_\sigma) = \sigma-1$. Kernel of the map is $C$ and thus we get an exact sequence
$$0 \ra C \xrightarrow{\subset} C(\varphi) \xrightarrow{\alpha} I_G \ra 0$$
We will show that $\forall r, H_T^r(G,C(\varphi))=0$ and therefore we get two exact sequences
\begin{align*}
& 0 \ra I_G \ra \ZZ[G] \ra \ZZ \ra 0 \\
& 0 \ra C \ra C(\varphi) \ra I_G \ra 0
\end{align*}
whose associated long exact sequence of Tate cohomology allows dimension shifting twice:
$$H_T^r(G,\ZZ) \cong H_T^{r+1}(G,I_G) \cong H_T^{r+2}(G,C)$$
which is our desired isomorphism. 

Before we show $\forall r, H_T^r(G,C(\varphi))=0$, we first show that $H^1(H,C(\varphi)) = H^2(H,C(\varphi)) = 0$. Consider 
$$0 \ra C \ra C(\varphi) \ra I_G \ra 0$$
and a part of its associated long exact sequence
\begin{align*}
0 \ra H^1(H,C(\varphi)) \ra H^1(H, I_G) \ra H^2(H,C) \xrightarrow 0 H^2(H,C(\varphi)) \ra 0
\end{align*}
Here the leftmost term $H^1(H,C)=0$ by hypothesis and the rightmost term $H^2(H,I_G)\cong H^1(H,\ZZ)=0$ by dimension shifting and Proposition~\ref{integer cohomology}. The map $H^2(H,C) \rightarrow H^2(H,C(\varphi))$ induced by inclusion $C \rightarrow C(\varphi)$ is a zero map because the generator $[\varphi]$ maps to a coboundary; $\d (\sigma \mapsto x_\sigma) = \varphi$. Due to exactness we see that $H^2(H,C(\varphi))=0$. Since $H^1(H,I_G) \cong H^0(H,\ZZ) \cong \ZZ/|H|$ by dimension shifting and Proposition~\ref{integer cohomology}, $H^2(H,C) \cong \ZZ/|H|$ too, and $H^1(H,I_G) \rightarrow H^2(H,C)$ is a surjection, it is a bijection and therefore the kernel $H^1(H,C(\varphi))$ is zero too.

Now we obtain $H_T^r(G,M)=0$ by the following general result, and Tate's theorem is proven.

\begin{proposition}
For a finite group $G$ and a $G$-module $M$, if $\forall H \le G, H^1(H,M)=H^2(H,M)=0$ then
$$\forall r, H_T^r(G,M)=0$$
\end{proposition}
\begin{proof}

\textbf{Case for solvable $G$:}

If $G$ is solvable, then there is a proper subgroup $H$ such that $G/H$ is cyclic. Assume induction hypothesis on order of $G$. Then $\forall r, H_T^r(H,M)=H_T^r(G/H,M^H)=0$. Consider the inflation-restriction sequence for $r\ge 1$:
\begin{align*}
& 0 \ra H_T^r(G/H, M^H) \xrightarrow{\on{Inf}} H_T^r(G,M) \xrightarrow{\on{Res}} H_T^r(H,M) = 0 \\
\implies & H_T^r(G/H, M^H) \cong H_T^r(G,M)
\end{align*}
By assumption, $H_T^1(G,M) = 0 = H_T^2(G,M)$ and thus
\begin{align*}
H_T^1(G/H,M^H) \cong H_T^1(G,M) = 0 = H_T^2(G,M) \cong H_T^2(G/H, M^H)
\end{align*}
By periodicity of cohomology of cyclic $G/H$,
$$ \forall r \in \ZZ, H_T^r(G/H, M^H)=0$$
and thus
\begin{align*}
H_T^r(G,M) \cong H_T^r(G/H,M^H) = 0 = H_T^2(G/H,M^H) \cong H_T^2(G,M)
\end{align*}
and therefore $r \ge 1 \implies H_T^r(G,M)=0$.

Now we show that $H_T^0(G,M) = M^G / \on{Nm}_G(M) = 0$. Suppose $x \in M^G$. Because $H_T^0(G/H, M^H) = 0$, there is $y \in M^H \subseteq M$ such that $\on{Nm}_{G/H}(y)=x$. Because $H_T^0(H,M)=0$, there exists a $z \in M$ such that $\on{Nm}_H (z) = y$. Thus $\on{Nm}_G(z) = (\on{Nm}_{G/H} \circ \on{Nm}_H)(z) = \on{Nm}_{G/H}(y) = x \implies H_T^0(G,M)=0$.

Tensor the exact sequence $0 \ra I_G \ra \ZZ[G] \ra \ZZ \ra 0$ with $M$ to get exact
\begin{align*}
0 \ra I_G \otimes_\ZZ M \ra \ZZ[G] \otimes_\ZZ M \ra M \ra 0
\end{align*}
Let's apply dimension shifting to this. Recall that $\forall H \le G, r\in \ZZ, H_T^r(H, \ZZ[G] \otimes_\ZZ M) = 0$ because $H_T^r(H, \ZZ[G] \otimes_\ZZ M) = H_T^r(H, (\bigoplus_{|G/H|} \ZZ[H]) \otimes_\ZZ M) = \bigoplus_{|G/H|} H_T^r(H,\ZZ[H] \otimes_\ZZ M) = \bigoplus 0 = 0$.

In particular, $I_G \otimes_\ZZ M$ satisfies the hypothesis of the theorem, and so $\forall r \ge 0, H_T^r(G,I_G \otimes_\ZZ M)=0$. But $H_T^{-1}(G,M) \cong H_T^0(G,I_G \otimes_\ZZ M)=0$ gives $H_T^{-2}(G,M)=0, H_T^{-3}(G,M)=0$, and so on.

\textbf{Case for any $G$:} Sylow $p$-subgroup is solvable and we factor through Cor and Res.

Suppose $G$ is any finite group. Find a Sylow $p$-subgroup $G_p \le G$. Since every finite $p$-group is nilpotent, and every nilpotent group is solvable, the above argument shows that $H_T^r(G_p,M)=0$ for all $r$. Now restriction map $H_T^r(G,M) \ra H_T^r(G_p,M)$ is an injection on the $p$-primary component (Cor$\circ$Res is multiplication by $[G:G_p]$, which is prime to $p$.). As $\text{Cor}\circ\text{Res}$ factors through $H_T^r(G_p,M)=0$, every $p$-primary component of $H_T^r(G,M)$ is zero. Thus $H_T^r(G,M)=0$.
\end{proof}

\end{proof}

\subsection{Proof Part II: $H^2(L/K)$ is cyclic (unramified case)}

We must prove that $H^2(\Gal(L/K), L^\times) \cong \frac1{[L:K]}\ZZ/\ZZ$ in order to apply Tate's theorem to prove local reciprocity. We'll first establish this for finite unramified $L/K$ and then extend the proof to ramified $L/K$.

For a finite unramified extension $L/K$ of local fields, consider the following composition of maps:
\begin{align*}
H^2(L/K) \xrightarrow{\on{ord}_L} H^2(G,\ZZ) \xleftarrow{\delta} H^1(G,\QQ/\ZZ) \cong \Hom(G, \QQ/\ZZ) \cong \frac1{[L:K]}\ZZ/\ZZ
\end{align*}
where we use the exact sequences 
\begin{align*}
& 0 \ra U_L \ra L^\times \xrightarrow{\on{ord}_L} \ZZ \ra 0\\
& 0 \ra \ZZ \ra \QQ \ra \QQ/\ZZ \ra 0
\end{align*}
for the first two maps. We already know that the second map ($\delta$) is an isomorphism, and we will soon show that the first map is an isomorphism. The third map is isomorphism because $G$ acts on $\QQ/\ZZ$ trivially. The fourth map is an isomorphism because $\Hom(G,\QQ/\ZZ)$ is determined by where a generator of $G$ (Frobenius map) maps to in $\QQ/\ZZ$, which has to be a multiple of $\frac1{[L:K]}$ due to torsion. Thus showing that $H^2(L/K)$ is finite cyclic of order $[L:K]$ will be complete by showing that cohomology of $U_L$ vanishes.

\begin{proposition}
For a finite unramified extension $L/K$,
$$H_T^r(G,U_L) = 0$$
In particular, the norm map restricted to unit group $\Nm_{L/K} :U_L \ra U_K$ is surjective.
\end{proposition}
\begin{proof}
$G$ is a cyclic group since $L/K$ is unramified. Since Tate cohomology of a cyclic group is 2-periodic, it suffices to show that $H_T^0(G, U_L)= H_T^1(G,U_L)=0$. Since $L/K$ is unramified, a uniformizer $\pi$ of $K$ is also a uniformizer of $L$. This implies that uniformizer is fixed by Galois element, such that we have a decomposition of $G$-modules $L^\times \cong U_L \times \ZZ$, and therefore
$$0 = H^1(G,L^\times) \cong H^1(G, U_L) \times H^1(G,\ZZ) \implies H^1(G,U_L) = 0$$

Thus it remains to prove that $H_T^0(G,U_L)=U_L^G / \on{Nm}_G(U_L) = U_K / \on{Nm}_{L/K}(U_L)=0$, i.e. $\on{Nm}_{L/K} : U_L \ra U_K$ is surjective. We prove this by passing to quotient fields of unit groups.

Denoting $\ell, k$ by residue fields of $L,K$ and $U_K^{(m)} = 1 + \mf{m}_K^m$, observe that $U_L / U_L^{(1)} \cong \ell^\times$ and $U_L^{(m)} / U_L^{(m+1)} \cong \ell^+$ as $G$-modules because $u \mapsto u \text{ mod }\mf{m}_L$ and $1+ a \pi^m \mapsto a\text{ mod } \mf{m}_L$ induce the isomorphisms.

We claim that trace and norm maps $\on{Nm}: \ell^\times \ra k^\times$ and $\on{Tr}: \ell^+ \ra k^+$ are surjective. This follows from cohomology computation 
$$H_T^r(\Gal(\ell/k),\ell^\times)=H_T^r(\Gal(\ell/k),\ell^+)=0$$
The first follows from Hilbert 90: $H_T^1(\Gal(\ell/k),\ell^\times) = 0$ and $H_T^0(\Gal(\ell/k),\ell^\times)=0$ too due to the Herbrand quotient being 1 for finite module over finite cyclic group. The second is due to normal basis theorem and periodicity of Tate cohomology for a finite cyclic group.

Now consider the commutative diagrams
\[\begin{tikzcd}
U_L \ar[r] \ar[d,"\on{Nm}"] & \ell^\times \ar[d,"\on{Nm}"] \\
U_K \ar[r] & k^\times
\end{tikzcd}
\begin{tikzcd}
U_L^{(m)} \ar[r] \ar[d,"\on{Nm}"] & \ell^+ \ar[d,"\on{Tr}"] \\
U_K^{(m)} \ar[r] & k^+
\end{tikzcd}\]
Given $u \in U_K$, we will construct $v \in U_L$ with $\Nm(v)=u$ by passing to quotients and taking a limit. By surjectivity we can find $v_0 \in U_L$ such that $u / \Nm(v_0) \in U_K^{(1)}$. By surjectivity of trace $\ell^+ \ra k^+$, we can find $v_1 \in U_L^{(1)}$ such that $ u / \Nm(v_0v_1) \text{ (mod $U_K^{(2)}$)}$. Continuing this way, we can find $v_0, \cdots v_m$ such that $u / \Nm(v_0 \cdots v_n) \in U_K^{(m+1)}$. Now let
$$v = \lim_{m \ra \infty} v_1 \cdots v_m$$
(the limit converges since $v_1 \cdots v_m - v_1 \cdots v_s $ for large $m,s$ has small value in the local field) and we have $u / \Nm(v) \in \bigcap_{m=1}^\infty U_K^{(m)} = 1$ and thus $u = \Nm(v)$.
\end{proof}

Now we have isomorphisms $\on{inv}_{L/K}: H^2(L/K) \ra \frac1{[L:K]}\ZZ/\ZZ$. Let $E/L/K$ finite unramified extensions. By regarding $\frac1{[L:K]}\ZZ/\ZZ$ as a subgroup of $\QQ / \ZZ$, the following diagram commutes:
\[\begin{tikzcd}
H^2(L/K) \ar[rd,"\on{inv}_{L/K}"] \ar[rr,"\on{Inf}"] & & H^2(E/K) \ar[ld,"\on{inv}_{E/K}"] \\
& \QQ/\ZZ & 
\end{tikzcd}\]
and therefore by Proposition~\ref{profinite limit} we have the following isomorphism:
\begin{align*}
\on{inv}_K : H^2(K^{\on{ur}} / K) \ra \QQ/\ZZ
\end{align*}

We call $(\on{inv}_{L/K})^{-1}(\frac1{[L:K]}) = u_{L/K} \in H^2(L/K)$ the fundamental class of $L/K$. We can compute $u_{L/K}$ and even the local reciprocity map explicitly.

\begin{proposition}
$u_{L/K} \in H^2(L/K)$ is represented by the cocycle 
$$\varphi(\sigma^i, \sigma^j) = \begin{cases} 1 & \text{ if } i+j \le n-1 \\ \pi & \text{ if }i+j>n-1 \end{cases}$$
\end{proposition}
\begin{proof}
Consider uniformizer $\pi \in K$, which is also a prime of $L$. $H^r(G,L^\times) \cong H^r(G,U_L) \oplus H^r(G,\ZZ)$. Take a generator $\sigma$ of $G$. Choose $f \in H^1(G, \QQ/\ZZ) = \Hom(G, \QQ/\ZZ)$ such that $f(\sigma^i) = \frac in \text{ mod }\ZZ$. Define a lift of $f$ by $\tilde f : \sigma^i \mapsto \frac in$ ($\tilde f \in \Hom(G,\QQ)$). Then
\begin{align*}
(\delta \tilde f) (\sigma^i, \sigma^j) := \sigma^i \tilde f(\sigma^j) - \tilde f(\sigma^i \sigma^j) + \tilde f (\sigma^i) = \begin{cases} 0 & \text{if $i+j \le n-1$} \\ 1 & \text{if $i+j > n-1$} \end{cases}
\end{align*}
\end{proof}

\begin{proposition}\label{unramified lrl explicit}
In the map $G \cong H_T^{-2}(G,\ZZ) \ra H_T^0(G,L^\times) \cong \frac{K^\times}{\on{Nm}(L^\times)}$, the Frobenius element $\sigma \in G$ maps to the class of $\pi \text{ mod }\on{Nm}(L^\times)$.
\end{proposition}
\begin{proof}
As seen earlier, $H^{-2}_T(G,\ZZ) \cong H_T^{-1}(G,I_G) = I_G/I_G^2$, $\sigma \mapsto [\sigma-1]$. The boundary map $H_T^{-1}(G, I_G) \cong H_T^0(G,L^\times)$ is given by the snake lemma. 

The element $(\sigma-1)+I_G^2$ is the image of $x_\sigma + I_G L^\times (\varphi)$ in $L^\times(\varphi)_G$. $\on{Nm}_G(x_\sigma + I_G L^\times (\varphi))$ is the product of
\begin{align*}
x_\sigma &= x_\sigma \\
\sigma x_\sigma &= x_{\sigma^2}x_\sigma^{-1} \varphi(\sigma,\sigma) \\
\sigma^2 x_\sigma &= x_{\sigma^3}x_{\sigma^2}^{-1} \varphi(\sigma^2,\sigma) \\
& \vdots \\
\sigma^{n-1} x_\sigma &= x_1x_{\sigma^{n-1}}^{-1} \varphi(\sigma^{n-1}, \sigma)
\end{align*}
(where $x_1=\varphi(1,1)=1$ here) and thus
$$\on{Nm}_G(x_\sigma) = \varphi(\sigma^1, \sigma) \cdots \varphi(\sigma^{n-1}, \sigma)=\pi$$
by the earlier calculation of $\varphi$.
\end{proof}

Since all units of $K$ are norms of $L$, $[\pi]$ and $\on{Nm}(L^\times)$ is independent of the choice of $\pi$. But $L^\times(\varphi)$ and the map depend on $\sigma \in G$.

We prove the following assertion for usage in the next section.

\begin{proposition}\label{unramified qz}
Suppose $L/K$ is a finite extension of degree $n$. The following commutes:
\[\begin{tikzcd}
H^2(K^{\on{ur}}/K) \ar[r,"\text{Res}"] \ar[d,"\on{inv}_K"] & H^2(L^{\on{ur}}/L) \ar[d,"\on{inv}_K"] \\
\QQ/\ZZ \ar[r, "\cdot n"] & \QQ/\ZZ
\end{tikzcd}\]
(here $\on{Res}$ is defined using compatible maps $\Gal(L^{\on{ur}}/L) \rightarrow \Gal(K^{\on{ur}}/K)$ and $K^{\on{ur} \times} \rightarrow L^{\on{ur} \times}$)
\end{proposition}
\begin{proof}
The maximal unramified extension of a local field is obtained by adjoining all $m$th root of unity where $m$ is coprime to characteristic of the base field. Thus $L^{\on{ur}} = L \cdot K^{\on{ur}}$. Thus we have injectivity of the restriction map $\Gal(L^{\on{ur}}/L) \ra \Gal(K^{\on{ur}}/K)$. 

Consider 
\[\begin{tikzcd}
H^2(K^{\on{ur}}/K) \ar[r,"\on{ord}_K"] \ar[d,"\on{Res}"] & H^2(\Gal(K^{\on{ur}}/K),\ZZ) \ar[d,"e \cdot \on{Res}"] & H^1(\Gal(K^{\on{ur}}/K),\QQ/\ZZ) \ar[l,"\delta"] \ar[r,"\cong"] \ar[d,"e \cdot \on{Res}"] & \QQ/\ZZ \ar[d,"n = ef"] \\
H^2(L^{\on{ur}}/L) \ar[r,"\on{ord}_L"] & H^2(\Gal(L^{\on{ur}}/L),\ZZ) & H^1(\Gal(L^{\on{ur}}/L),\QQ/\ZZ) \ar[l,"\delta"] \ar[r,"\cong"] & \QQ/\ZZ 
\end{tikzcd}\]
We now check commutativity. The first square comes from the commutative square
\[\begin{tikzcd}
(K^{\on{ur}})^\times \ar[r,"\on{ord}_K"] \ar[d] & \ZZ \ar[d,"e"] \\
(L^{\on{ur}})^\times \ar[r,"\on{ord}_L"] & \ZZ
\end{tikzcd}\]
The second square is commutative because restriction map commutes with connecting homomorphism. To check commutativity of the third square, first note that $H^1(\Gal(K^{\on{ur}}/K),\QQ/\ZZ) \cong \Hom(\Gal(K^{\on{ur}}/K),\QQ/\ZZ)$ and the row isomorphisms are obtained by evaluation at Frobenius generator. $\on{Frob}_L|_K = (\on{Frob}_K)^f$ because Frobenius is determined by its action on the residue fields. Thus the third square commutes too.
\end{proof}

\subsection{Proof Part III: $H^2(L/K)$ is cyclic (general)}

We extend the previously obtained results on unramified extensions to ramified extensions and complete the proof of local reciprocity.

In this section, we denote by $\bar K$ the separable closure of $K$ (which is just algebraic closure in the case of characteristic 0).

\begin{theorem}
For every local field, there is a canonical isomorphism
$$\on{inv}_{K} : H^2(\bar K/K) \cong \QQ/\ZZ $$ and if $L/K$ is finite of degree $n$, then we have a commutative diagram with vertical maps being isomorphisms:
\[\begin{tikzcd}
0 \ar[r] & H^2(L/K) \ar[r] \ar[d,"\on{inv}_{L/K}"] & H^2(\bar K / K) \ar[r,"\on{Res}"] \ar[d, "\on{inv}_K"] & H^2(\bar K / L) \ar[d, "\on{inv}_L"] \\
0 \ar[r] & \frac1n \ZZ/\ZZ \ar[r]& \QQ/\ZZ \ar[r,"\cdot n"] & \QQ/\ZZ
\end{tikzcd}\]
where in particular
$$\on{inv}_{L/K}: H^2(L/K) \cong \frac1n \ZZ / \ZZ$$
\end{theorem}

\begin{proof}
We first prove $H^2(L/K) \cong \frac1n \ZZ / \ZZ$ by embedding it into $H^2(\bar K / K)$ and using the unramified case. Observe the following commutative diagram:
\[\begin{tikzcd}
0 \ar[r] & \ker(\on{Res}) \ar[r] \ar[d] & H^2(K^{\on{ur}}/K) \ar[r,"\on{Res}"] \ar[d,"\on{Inf}"] & H^2(L^{\on{ur}} / L) \ar[d,"\on{Inf}"] \\
0 \ar[r] & H^2(L/K) \ar[r] & H^2(\bar K / K) \ar[r,"\on{Res}"] & H^2(\bar L / L)
\end{tikzcd}\]
Here, we use Hilbert 90 for infinite extensions to apply Inflation-Restriction sequence twice and see that the bottom row is exact and the vertical $\on{Inf}$ maps are injections. Thus $\ker(\on{Res}) \hookrightarrow H^2(L/K)$ too, and by Proposition~\ref{unramified qz}, $\ker(\on{Res}) \cong \frac 1n \ZZ / \ZZ$, there is a copy of $\frac1n \ZZ/\ZZ$ in $H^2(L/K)$. Therefore it now suffices to prove $|H^2(L/K)| \le n$ to prove $H^2(L/K) \cong \frac1n \ZZ / \ZZ$.

We prove $H^2(L/K) \cong \frac1n\ZZ / \ZZ$ by induction on degree $[L:K]$ and using solvability. Proposition~\ref{local solvable} tells us that $\Gal(L/K)$ is solvable, and thus we may find a an intermediate extension $K'$ with $\Gal(K'/K)$ finite cyclic. By Hilbert 90, we can use the inflation-restriction sequence:
$$0 \ra H^2(K'/K) \xrightarrow{\on{Inf}} H^2(L/K) \xrightarrow{\on{Res}} H^2(L/K')$$
It can be shown that $|H^2(L/K)|=[L:K]$ for finite cyclic extension $L/K$ of local fields, the proof of which we postpone. Assuming this fact, we see by the induction hypothesis that $|H^2(L/K)| \le |H^2(K'/K)| \cdot |H^2(L/K')| = n$. This along with the earlier observation that $\frac1n\ZZ / \ZZ$ injects into $H^2(L/K)$ tells us that the injection is an isomorphism.

Now $H^2(\bar K / K) \cong \QQ/\ZZ$ holds by `sandwiching' since $\forall L/K$ finite, $H^2(L/K) \cong \ker(\on{Res}) \xrightarrow{\on{Inf}} H^2(\bar K / K) \cong \lim_\rightarrow H^2(L/K) = \bigcup H^2(L/K)$.

We complete the proof by proving the following result that we previously postponed:
\begin{lemma}\label{cyclic brauer}
For a cyclic extension of local fields $L/K$ of degree $n$, $h(U_L)=1$ and $h(L^\times)=n$. In particular, $|H^2(L/K)|=n$
\end{lemma}
\begin{proof}
To prove $h(U_L)=1$, we'll later construct an open subgroup of $U_L$ such that $\forall r>0, H^r(G,V)=0$ in particular satisfying $h(V)=1$. Once this is done, we get an exact sequence
$$0 \ra V \ra U_L \ra U_L / V \ra 0$$
from which we deduce that $h(U_L) = h(V) h(U_L/V)$, but since $U_L$ is compact, $U_L / V$ is finite, we see that $h(U_L/V) = 1$ and therefore $h(U_L)=1 \cdot 1 = 1$. 

To prove $h(L^\times)=n$, we simply use the exact sequence
$$1 \ra U_L \ra L^\times \xrightarrow{\on{ord}_L} \ra \ZZ \ra 1$$
from which we deduce $h(L^\times)=h(U_L) h(\ZZ) = h(\ZZ) = \frac{|H_T^0(G,\ZZ)|}{|H_T^1(G,\ZZ)|} = \frac{|H_T^0(G,\ZZ)|}{|\Hom(G,\ZZ)|} = |H_T^0(G,\ZZ)| = |\ZZ/n\ZZ| = n$ and we're done.

Now we construct $V$, but first we construct $V_0 \subseteq \mathcal O_L$ with vanishing cohomology. Take a normal basis of $L$ over $K$, $\{x_\tau | \tau \in G\}$. $x_\tau$ have a common denominator $d$ in $\mc O_K$. By replacing $x_\tau$ by $d \cdot x_\tau$, we may suppose they lie in $\mc O_L$. Letting $V_0 = \bigoplus \mc O_K x_\tau \subseteq \mc O_L$, $V_0 \cong \mc O_K[G] \cong \Hom_{\ZZ}(\ZZ[G], \mc O_K)$ and thus $V_0$ has vanishing cohomology.

$V$ is obtained by moving $V_0$ by exponential map (here we only prove the characteristic zero case). The exponential map $\exp x = \sum_{m=0}^\infty \frac{x^m}{m!}$ converges when $\on{ord}(x) > \frac{\on{ord}(p)}{p-1}$ and moves an open neighborhood of $0$ to an open neighborhood of $1$ isomoprhically. Let $V = \exp(\pi^M V_0)$ where $M$ is large enough so that $\exp x$ is well-defined.
\end{proof}
\end{proof}

This proves that we have an isomorphism
$$\phi_{L/K}^{-1} : \Gal(L/K) \cong \frac{K^\times}{\Nm_{L/K}(L^\times)} $$
To prove the full reciprocity law, we will need the following compatibility result:

\begin{proposition}
For finite Galois extensions of local fields $E/L/K$, the following commutes:
\[\begin{tikzcd}
& K^\times \ar[ld,"\tilde \phi_{L/K}"]\ar[rd,"\tilde \phi_{E/K}"] & \\
\Gal(L/K)^{\on{ab}}  \ar[rr,"\text{restrict}"] & & \Gal(E/K)^{\on{ab}} 
\end{tikzcd}\]
where $\tilde \phi_{L/K}:K^\times \ra \Gal(L/K)^{\on{ab}}$ is composition of $\phi_{L/K}$ with projection $K^\times \ra K^\times/\Nm_{L/K}(L^\times)$.
\end{proposition}

We thus obtain the local reciprocity map:
$$\phi_K: K^\times \rightarrow \Gal(L/K)^{\on{ab}}$$
We also showed earlier in Proposition~\ref{unramified lrl explicit} that inverse of local reciprocity maps Frobenius element to the uniformizer, so this completes the proof of local reciprocity law.

\subsection{Description of Local Reciprocity}

Here we describe the (inverse of) local reciprocity map in two different ways and show that the two are equal. The first description is as a composition of two connecting homomorphisms, as obtained in Tate's theorem (Theorem~\ref{tate}). Another description is obtained by taking cup product with the fundamental class.

The following is an explicit description of the map obtained in Tate's theorem.
\begin{proposition}\label{lr explicit}
For finite Galois extension of local fields $L/K$ and Galois group $G = \Gal(L/K)$, inverse of local reciprocity map
$$G^{\on{ab}} \rightarrow \frac{K^\times}{\Nm_{L/K}(L^\times)}$$
obtained as a composition of isomorphisms
\begin{align*}
G^{\on{ab}} \rightarrow H_T^{-2}(G,\ZZ) \rightarrow H_T^{-1}(G,I_G) \rightarrow H_T^0(G,L^\times) = \frac{K^\times}{\Nm_{L/K}(L^\times)}
\end{align*}
is given by
\begin{align*}
[\sigma] \mapsto \left[ \prod_{\tau \in G} \varphi(\tau,\sigma) \right]
\end{align*}
where $[\varphi]$ is a generator of $H^2(G, L^\times)$.
\end{proposition}
\begin{proof}
The isomorphism $G^{\on{ab}} \rightarrow H_T^{-2}(G,\ZZ)$ is a composition of isomorphisms $G^{\on{ab}} \rightarrow I_G / I_G^2 = H_T^{-1}(G,I_G) \rightarrow H_T^{-2}(G,\ZZ)$, and thus composing this with isomorphism $H_T^{-2}(G,\ZZ) \rightarrow H_T^{-1}(G,I_G)$ is the same as the isomorphism
\begin{align*}
& G^{\on{ab}} \rightarrow H_T^{-1}(G,I_G) \\
& [\sigma] \mapsto [\sigma -1]
\end{align*}
We also have
\begin{align*}
& H_T^{-1}(G,I_G) \rightarrow H_T^0(G,L^\times) \\
& [a] \mapsto [\Nm_G (\tilde a)]
\end{align*}
where $\tilde a \in L^\times (\varphi)$ is a lift of $a$ under the surjection $L^\times (\varphi) \xrightarrow{\alpha} I_G$. As $\alpha$ is given by sending all elements of $L^\times$ to zero and sending $x_\sigma$ to $\sigma - 1$, composition of the two maps above gives
\begin{align*}
& G^{\on{ab}} \rightarrow H_T^0(G,L^\times) \\
& [\sigma] \mapsto [\Nm_G(\sigma)]
\end{align*}
and we further compute:
\begin{align*}
[\Nm_G(\sigma)] &= \left[ \sum_{\tau \in G} \tau \cdot x_\sigma \right] = \left[ \sum_{\tau \in G} x_{\tau \sigma} - x_\tau + \varphi(\tau, \sigma) \right] = \left[ \sum_{\tau \in G} x_{\tau \sigma} - x_\tau \right] + \left[ \sum_{\tau \in G} \varphi(\tau, \sigma) \right] = \left[ \prod_{\tau \in G} \varphi(\tau, \sigma) \right]
\end{align*}
by definition of $G$-action on the splitting module $L^\times(\varphi)$ and noting that $\sum_{\tau \in G} x_{\tau \sigma} = \sum_{\tau \in G} x_{\tau}$. Here, at the very end, sum switched to product as we use multiplicative notation for elements of $L^\times$. 
\end{proof}

Now we compute local reciprocity using cup product of Tate cohomology. The assertion regards cup product of elements from $H^2$ and $H^{-2}$, and for such computation we will first need computation of cup product for ($H^1$ and $H^{-1}$) , ($H^1$ and $H^{-2}$), and ($H^0$ and $H^p$, $p>0$). 

\begin{proposition}
Cup product of elements $[a_p] \in H_T^p(G,A)$ and $[b_{0}] \in H_T^0(G,B)$ is given by
$$[a_p] \smile [b_0] = [a_p \otimes b_0]$$
\end{proposition}
\begin{proof}
The maps defined this way satisfy the conditions required for cup product, and by uniqueness of cup product these are the cup product maps.
\end{proof}

\begin{proposition}
Cup product of elements $[a_1] \in H_T^{1}(G,A)$ and $[b_{-1}] \in H_T^{-1}(G,B)$ is given by
\begin{align*}
[a_1] \smile [b_{-1}] = \sum_{\tau \in G} a_1(\tau) \otimes \tau b_{-1}
\end{align*}
\end{proposition}
\begin{proof}
We'll find $a_0$ such that $[a_1] = \delta [a_0]$ to calculate cup product using connecting homomorphisms. Since $H_T^1(G, \ZZ[G] \otimes A)=0$, we can find $a_0' \in \ZZ[G] \otimes A$ such that $$\d a_0' = (\sigma \mapsto \sigma a_0' - a_0') = a_1$$
Let $a_0 \in J_G \otimes A$ be the image of $a_0'$ under map $\ZZ[G] \otimes A \rightarrow J_G \otimes A$. Then $[a_1] = \delta[a_0]$ and therefore
\begin{align*}
[a_1] \smile [b_{-1}] &= \delta[a_0] \smile [b_{-1}] = \delta([a_0] \smile [b_{-1}]) = \delta([a_0 \otimes b_{-1}]) = [\Nm_G(a_0' \otimes b_{-1} ) ] \\
&= \left[ \sum_{\tau \in G} \tau a_0' \otimes \tau b_{-1} \right] \\
&= \left[ \sum_{\tau \in G} (a_1(\tau) + a_0') \otimes \tau b_{-1} \right] \\
&= \left[ \sum_{\tau \in G} (a_1(\tau) \otimes \tau b_{-1}) \right] + \left[ N_G(a_0' \otimes \Nm_G b_{-1}) \right] \\
&= \left[ \sum_{\tau \in G} (a_1(\tau) \otimes \tau b_{-1}) \right] + \left[ a_0' \otimes 0 \right] \\
&= \left[ \sum_{\tau \in G} (a_1(\tau) \otimes \tau b_{-1}) \right]
\end{align*}
\end{proof}

\begin{proposition}
Cup product of elements $[a_1] \in H_T^1(G,A)$ and $[\sigma] \in G^{\on{ab}} \cong H_T^{-2}(G,A)$ is given by
$$[a_1] \smile [\sigma] = [a_1(\sigma)] \in H^{-1}(G,A)$$
\end{proposition}
\begin{proof}
Consider exact sequence
$$0 \ra A \otimes I_G \ra A \otimes \ZZ[G] \ra A \ra 0$$
Since cohomology of $A \otimes \ZZ[G]$ vanishes, connecting homomorphism $\delta : H_T^{-1}(G,A) \ra H_T^0(G,A \otimes I_G)$ is an isomorphism. Thus it suffices to show $\delta([a_1] \smile [\sigma]) = \delta([a_1 (\sigma)])$. 

$\delta([a_1(\sigma)])$ is
\begin{align*}
\delta([a_1(\sigma)]) = \left[ \Nm_G(a_1(\sigma) \otimes 1) \right] = \left[ \sum_{\tau \in G} \tau a_1(\sigma) \otimes \tau \right]
\end{align*}
Also
\begin{align*}
\delta([a_1] \smile [\sigma]) & = -([a_1] \smile \delta[\sigma]) \\
& = -[a_1] \smile [\sigma - 1] \\
& = \left[ \sum_{\tau \in G} a_1(\tau) \otimes \tau(1 - \sigma) \right] \\
& = \left[ \sum_{\tau \in G} a_1(\tau) \otimes \tau - \sum_{\tau \in G} a_1(\tau) \otimes \tau \sigma \right] \\
& = \left[ \sum_{\tau \in G} a_1(\tau) \otimes \tau - \sum_{\tau \in G} (a_1(\tau \sigma) - \tau a_1(\sigma)) \otimes \tau \sigma \right] \\
& = \left[ \sum_{\tau \in G} \tau a_1(\sigma) \otimes \tau \sigma \right]
\end{align*}
Therefore
\begin{align*}
\delta \left( [a_1(\sigma)] - [a_1] \smile [\sigma] \right) &= \left[ \sum_{\tau \in G} \tau a_1(\sigma) \otimes \tau \right] - \left[ \sum_{\tau \in G} \tau a_1(\sigma) \otimes \tau \sigma \right] \\
&= \left[ \sum_{\tau \in G} \tau a_1(\sigma) \otimes \tau(1-\sigma) \right] \\
&= \left[ \Nm_G (a_1(\sigma) \otimes (1 - \sigma)) \right] \\
&= 0 \in \frac{(A \otimes I_G)^G}{\Nm_G (A \otimes I_G)} = H_T^0(G, A\otimes I_G)
\end{align*}
as desired.
\end{proof}

\begin{proposition}[Local Reciprocity as Cup Product]
Inverse of local reciprocity is given as cup product
\begin{align*}
& H_T^{-2}(G,\ZZ) \rightarrow H_T^0(G,L^\times) \\
& [\sigma] \mapsto [\varphi] \smile [\sigma]
\end{align*}
where $[\varphi]$ is a generator of $H_T^2(G,L^\times)$, i.e. the fundamental class.
\end{proposition}
\begin{proof}
By Proposition~\ref{lr explicit}, we already know that inverse of local reciprocity is given by $[\sigma] \mapsto [\prod_{\tau \in G} \varphi(\tau,\sigma)]$ and therefore it suffices to prove that
\begin{align*}
[\varphi] \smile [\sigma] = \left[ \prod_{\tau \in G} \varphi(\tau,\sigma) \right]
\end{align*}
To make use of connecting homomorphisms, we'll find $\eta$ such that $[\varphi]=\delta [\eta]$. Consider
$$0 \ra L^\times \ra \ZZ[G] \otimes L^\times \ra J_G \otimes L^\times \ra 0$$
and
\[\begin{tikzcd}
0 \ar[r] & \Hom_{\ZZ[G]}(P_1, L^\times) \ar[r]\ar[d] & \Hom_{\ZZ[G]}(P_1, \ZZ[G] \otimes L^\times) \ar[r]\ar[d] & \Hom_{\ZZ[G]}(P_1, J_G \otimes L^\times) \ar[r]\ar[d] & 0 \\
0 \ar[r] & \Hom_{\ZZ[G]}(P_2, L^\times) \ar[r] & \Hom_{\ZZ[G]}(P_2, \ZZ[G] \otimes L^\times) \ar[r] & \Hom_{\ZZ[G]}(P_2, J_G \otimes L^\times) \ar[r] & 0
\end{tikzcd}\]
Here $\varphi \in \Hom_{\ZZ[G]}(P_2, L^\times)$ and since $H^2(G,\ZZ[G] \otimes L^\times) = 0$, $\varphi = \d \eta_0$ when we regard $\varphi$ as an element of $\Hom_{\ZZ[G]}(P_2, \ZZ[G] \otimes L^\times)$. Further let image of $\eta_0$ under the map $\Hom(P_1, \ZZ[G] \otimes L^\times) \rightarrow \Hom(P_1, J_G \otimes L^\times)$ be $\eta$. Then by definition of connecting homomorphism, we have $[\varphi] = \delta [\eta]$. 

Therefore
\begin{align*}
[\varphi] \smile [\sigma] &= \delta[\eta] \smile [\sigma] = \delta([\eta] \smile [\sigma]) = \delta([\eta(\sigma)]) \\
&= \left[ \sum_{\tau \in G} \tau \eta_0(\sigma) \right] = \left[ \sum_{\tau \in G} ( \varphi(\tau, \sigma) + \eta_0(\tau \sigma) - \eta_0 (\tau) ) \right] = \left[ \prod_{\tau \in G} \varphi(\tau, \sigma) \right]
\end{align*}
where we used the fact that $\varphi (\tau, \sigma) = \d \eta_0(\tau,\sigma) = \tau \eta_0(\sigma) -\eta_0(\tau \sigma) + \eta_0 (\tau)$ and we switched to multiplicative notion at the end because we are working with group law of $L^\times$.
\end{proof}


\begin{thebibliography}{PTW02}
    \bibitem{milne_cft} Milne, James S. \textit{Class field theory}, 2011.
    \bibitem{milne_ant} Milne, James S. \textit{Algebraic number theory.} JS Milne, 2008.
    \bibitem{serre_lf} Serre, Jean-Pierre. \textit{Local fields.} Vol. 67. Springer Science \& Business Media, 2013.
\end{thebibliography}
\end{document}